\documentclass[12pt]{article} 

\usepackage{endnotes}
\let\footnote=\endnote

\usepackage{natbib}
\bibpunct[, ]{(}{)}{,}{a}{}{,}%
\usepackage{booktabs} 
\usepackage{array} 
\usepackage{paralist} 
\usepackage{verbatim} 

\usepackage{graphicx,subfigure}
\usepackage{caption}
\usepackage{epstopdf}

\usepackage{amsmath, amssymb}

\usepackage{multirow}
\usepackage{algorithm}
\usepackage{algorithmic}
\usepackage{mathtools}
\usepackage{amsthm}

\newtheorem{prop}{Proposition}[section]
\newtheorem{thm}{Theorem}[section]

\usepackage{rotating}
\usepackage{pdflscape}
\usepackage{url}
\graphicspath{{figures/}} 

\newcommand{\exclude}[1]{}

\DeclareMathOperator{\atan2}{atan2}

\usepackage{setspace}
\usepackage[margin=0.7in]{geometry}
\usepackage{graphicx} 

\allowdisplaybreaks

\begin{document}




\title{Strong SOCP Relaxations for the Optimal Power Flow Problem}


\author{Burak Kocuk, Santanu S. Dey, X. Andy Sun  
}
	



\maketitle

\abstract{%
This paper proposes three strong second order cone programming (SOCP) relaxations for the AC optimal power flow (OPF) problem. These three relaxations are incomparable to each other and two of them are incomparable to the standard SDP relaxation of OPF. Extensive computational experiments show that these relaxations have numerous advantages over existing convex relaxations in the literature: (i) their solution quality is extremely close to that of the SDP relaxations (the best one is within $99.96\%$ of the SDP relaxation on average for all the IEEE test cases) and consistently outperforms previously proposed convex quadratic relaxations of the OPF problem, (ii) the solutions from the strong SOCP relaxations can be directly used as a warm start in a local solver such as IPOPT to obtain a high quality feasible OPF solution, and (iii) in terms of computation times, the strong SOCP relaxations can be solved an order of magnitude faster than standard SDP relaxations. For example, one of the proposed SOCP relaxations together with IPOPT produces a feasible solution for the largest instance in the IEEE test cases (the 3375-bus system) and also certifies that this solution is within $0.13\%$ of global optimality, all this computed within $157.20$ seconds on a modest personal computer. Overall, the proposed strong SOCP relaxations provide a practical approach to obtain feasible OPF solutions with extremely good quality within a time framework that is compatible with the real-time operation in the current industry practice.
}%

\section{Introduction}\label{sec:intro}

Optimal Power Flow (OPF) is a fundamental optimization problem in electrical power systems analysis \citep{Carpentier}. There are two challenges in the solution of OPF{.} First, it is an operational level problem solved every few minutes, hence the computational budget is limited. Second, it is a nonconvex optimization problem on a large-scale power network of thousands of buses, generators, and loads. The importance of the problem and the aforementioned difficulties have produced a rich literature, see e.g. \citep{Momoh11999, Momoh21999, Frank12012, Frank22012, FERC2012}.

Due to these challenges, the current practice in the electricity industry is to use the so-called DC OPF approximation \citep{FERC2011}. In contrast, the original nonconvex OPF is usually called the AC (alternating current) OPF. DC OPF is a linearization of AC OPF by exploiting some physical properties of the power flows in typical power systems, such as tight bounds on voltage magnitudes at buses and small voltage angle differences between buses. However, such an approximation completely ignores important aspects of power flow physics, such as the reactive power and voltage magnitude. To partially remedy this drawback, the current practice is to solve DC OPF and then to solve a set of power flow equations with the DC OPF solution to compute feasible reactive powers and voltages. However, it is clear such an approach cannot guarantee any optimality of the AC power flow solution obtained. To be concise, we will use OPF to denote AC OPF in the remainder of the paper. 

In order to solve the OPF problem, the academic literature has focused on improving nonlinear optimization methods such as the interior point methods (IPM) to compute local optimal solutions, see e.g. \citep{Wu1994, Torres1998, Jabr2002, Zimmerman2007}. A well-known implementation of IPM tailored for the OPF problem is MATPOWER \citep{Matpower}. Although these local methods are effective in solving IEEE test instances, they do not offer any quantification of the quality of the solution.

In the recent years, much research interests have been drawn to the convex relaxation approach. In particular, the second-order cone programming (SOCP) and the semidefinite programming (SDP) relaxations are first applied to the OPF problem in \cite{Jabr06}, and \cite{Bai2008, Bai2009} and \cite{Lavaei12}. Among these two approaches, the SDP relaxation and its variations have drawn significant attention due to their strength. Since convex conic programs are polynomially solvable, the SDP relaxation offers an effective way for obtaining global optimal solutions to OPF problems whenever the relaxation is exact. Unfortunately, the exactness of the SDP relaxations can be guaranteed only for a restricted class of problems under some assumptions, e.g. radial networks \citep{Zhang12, bose2011optimal, bose2012quadratically} under load over-satisfaction \citep{sojoudi2012physics} or absence of generation lower bounds \citep{Lavaei14}, or lossless networks with cyclic graphs \citep{Zhang2013}. A comprehensive survey can be found in \cite{low2014convexi, low2014convexii}. When the SDP relaxation is not exact, it may be difficult to put a physical meaning on the solution.

A way to further strengthen the SDP relaxation is to solve a hierarchy of moment relaxation problems \citep{lasserre, Parrilo2003}. This approach is used in \cite{josz} to globally solve small-size problems, and is also used in \cite{molzahn2014} to obtain tighter lower bounds for larger problems of 300-bus systems. However due to the NP-hardness of the OPF problem \citep{Lavaei12}, in general the order of the Lasserre hierarchy required to obtain a global optimal solution can be arbitrarily large. Furthermore, even the global optimal objective function value is achieved, the solution matrices may not be rank one, which poses another challenge in terms of recovering an optimal voltage solution \citep{Lavaei14}. This indicates the computational difficulty of the SDP relaxation approach to practically solve real-world sized power networks with more than a thousand buses. For such large-scale OPF problems, a straightforward use of IPM to solve the SDP relaxation becomes prohibitively expensive. Interesting works have been done to exploit the sparsity of power networks as in \cite{jabr2012, molzahn2013, madani2014, molzahn2014, madani2015}. The underlying methodology utilizes techniques such as chordal graph extension, tree-width decomposition, and matrix completion, as proposed and developed in \cite{fukuda2001} and \cite{nakata2003}.

More recently, there is a growing trend to use computationally less demanding relaxations based on linear programming (LP) and SOCP to solve the OPF problem. For instance, linear and quadratic envelopes for trigonometric functions in the polar formulation of the OPF problem are constructed in \cite{coffrin2014, hijazi2013, coffrin2015}. In \cite{bienstock2014}, LP based outer approximations are proposed which are strengthened by incorporating several different types of valid inequalities.

This paper proposes new strong SOCP relaxations of the OPF problem and demonstrates their computational advantages over the SDP relaxations and previously described convex quadratic relaxations for the purpose of practically solving large-scale OPF problems. Our starting point is an alternative formulation for the OPF problem proposed in \cite{Exposito99} and \cite{Jabr06}. In this formulation, the nonconvexities are present in two types of constraints: one type is the surface of a rotated second-order cone, and the other type involves arctangent functions on voltage angles. The SOCP relaxation in \cite{Jabr06} is obtained by convexifying the first type of constraints to obtain SOCP constraints and completely ignoring the second type constraints. We refer to this relaxation as the \emph{classic SOCP relaxation} of the OPF problem. We prove that the standard SOCP relaxation of the rectangular formulation of OPF provides the same bounds as the classic SOCP relaxation. Therefore, if we are able to add convex constraints that are implied by the original constraints involving the arctangent function to the classic SOCP relaxation, then this could yield \emph{stronger relaxation than the classic SOCP relaxation that may also potentially be incomparable to (i.e. not dominated by nor dominates) the standard SDP relaxations}. In this paper, we propose three efficient ways to achieve this goal.

In the following, we summarize the key contributions of the paper.
\begin{enumerate}
	\item We theoretically analyze the relative strength of the McCormick (linear programming), SOCP, and SDP relaxations of the rectangular and alternative formulations of the OPF problem. As discussed above, this analysis leads us to consider strengthening the classic SOCP relaxation as a way forward to obtaining strong and tractable convex relaxations.
	\item We propose three efficient methods to strengthen the classic SOCP relaxation. 
	\begin{enumerate}
		\item In the first approach, we begin by reformulating the arctangent constraints as polynomial constraints whose degrees are proportional to the length of the cycles. This yields a bilinear relaxation of the OPF problem in extended space (that is by addition of artificial variables), where the new variables correspond to edges obtained by triangulating cycles. With this reformulation, we use the McCormick relaxation of the proposed bilinear constraints to strengthen the classic SOCP relaxation. The resulting SOCP relaxation is shown to be incomparable to the SDP relaxation. 
		\item In the second approach, we construct a polyhedral envelope for the arctangent functions in 3-dimension, which are then incorporated into the classic SOCP relaxation. This SOCP relaxation is also shown to be incomparable to the standard SDP relaxation.
		\item In the third approach, we strengthen the classic SOCP relaxation by dynamically generating valid linear inequalities that separate the SOCP solution from the SDP cone constraints over cycles. We observe that running such a separation oracle a few iterations already produces SOCP relaxation solutions very close to the quality of the full SDP relaxation. 
	\end{enumerate}
	\item We conduct extensive computational tests on the proposed SOCP relaxations and compare them with existing SDP relaxations \citep{Lavaei14} and quadratic relaxations \citep{coffrin2014, coffrin2015}. The computational results can be summarized as follows.
		\begin{enumerate}
			\item Lower bounds: The lower bounds obtained by the third proposed SOCP relaxation for all MATPOWER test cases from 6-bus to 3375-bus are on average within $99.96\%$ of the lower bounds of the SDP relaxation.  The other two proposed relaxation are also on average within $99.7\%$ of the SDP relaxation.
			\item Computation time: Overall, the proposed SOCP relaxations can be solved orders of magnitude faster than the SDP relaxations. The computational advantage is even more evident when a feasible solution of the OPF problem is needed. As an example, consider the largest test instance of the IEEE 3375-bus system. Our proposed SOCP relaxation together with IPOPT provides a solution for this instance and also certifies that this solution is within $0.13\%$ of global optimality, all computed in $157.20$ seconds on a modest personal computer.
			\item Comparison with other convex quadratic relaxation: The proposed SOCP relaxations consistently outperform the existing quadratic relaxation in \cite{coffrin2014}  and \cite{coffrin2015} on the test instances of typical, congested, and small angle difference conditions.
			\item Non-dominance with standard SDP relaxation: The computation also shows that the proposed SOCP relaxations are neither dominated by nor dominates the standard SDP relaxations. 
			\item Robustness: The proposed SOCP relaxations perform consistently well on IEEE test cases with randomly perturbed load profiles. 
		\end{enumerate}
\end{enumerate}

The paper is organized as follows. Section \ref{sec:opfintro} introduces the standard rectangular formulation and the alternative formulation of the OPF problem. Section \ref{Comparison Radial} compares six different convex relaxations for the OPF problem based on the rectangular and alternative formulations. Section \ref{section:meshed} proposes three ways to strengthen the classic SOCP relaxation. Section \ref{comp experiment} presents extensive computational experiments. We make concluding remarks in Section \ref{sec:conc}.

\section{Optimal Power Flow Problem}\label{sec:opfintro}
Consider a power network $\mathcal{N} = (\mathcal{B},\mathcal{L})$, where $\mathcal{B}$ denotes the node set, i.e., the set of buses, and $\mathcal{L}$ denotes the edge set, i.e., the set of transmission lines. Generation units (i.e. electric power generators) are connected to a subset of buses, denoted as $\mathcal{G}\subseteq \mathcal{B}$. We assume that there is electric demand, also called load, at every bus. The aim of the optimal power flow problem is to satisfy demand at all buses with the minimum total production costs of generators such that the solution obeys the physical laws (e.g., Ohm's Law and Kirchoff's Law) and other operational restrictions (e.g., transmission line flow limit constraints).

Let $Y \in \mathbb{C}^{|\mathcal{B}| \times |\mathcal{B}|}$ denote the nodal admittance matrix, which has components $Y_{ij}=G_{ij} + \mathrm{i}B_{ij}$ for each line $(i,j)\in\mathcal{L}$, and $G_{ii}=g_{ii}-\sum_{j\ne i} G_{ij}, B_{ii}=b_{ii}-\sum_{j\ne i} B_{ij}$, where $g_{ii}$ (resp. $b_{ii}$) is the shunt conductance (resp. susceptance) at bus $i\in\mathcal{B}$ and $\mathrm{i}=\sqrt{-1}$. 
Let $p_i^g, q_i^g$ (resp. $p_i^d, q_i^d$) be the real and reactive power output of the generator (resp. load) at bus $i$. The complex voltage (also called voltage phasor) $V_i$ at bus $i$ can be expressed either in the rectangular form as $V_i = e_i+\mathrm{i} f_i$ or in the polar form as $V_i = |V_i|(\cos\theta_i+\mathrm{i}\sin\theta_i)$, where $|V_i|^2=e_i^2 + f_i^2$ is the voltage magnitude and $\theta_i$ is the angle of the complex voltage. In power system analysis, the voltage magnitude is usually normalized against a unit voltage level and is expressed in per unit (p.u.). For example, if the unit voltage is 100kV, then 110kV is expressed as 1.1 p.u.. In transmission systems, the bus voltage magnitudes are usually restricted to be close to the unit voltage level to maintain system stability. 

With the above notation, the OPF problem is given in the so-called rectangular formulation:
\begin{subequations}\label{rect}
\begin{align}
\min  &\hspace{0.25em}  \sum_{i \in \mathcal{G}} C_i(p_i^g)  \label{obj} \\
  \mathrm{s.t.}   &\hspace{0.25em} p_i^g-p_i^d = G_{ii}(e_i^2+f_i^2) + \sum_{j \in \delta(i)}[ G_{ij}(e_ie_j+f_if_j) -B_{ij}(e_if_j-e_jf_i)]   & i& \in \mathcal{B} \label{activeAtBus} \\
  & \hspace{0.25em} q_i^g-q_i^d = -B_{ii}(e_i^2+f_i^2) + \sum_{j \in \delta(i)}[ -B_{ij}(e_ie_j+f_if_j) -G_{ij}(e_if_j-e_jf_i)]  & i& \in \mathcal{B} \label{reactiveAtBus} \\
  & \hspace{0.25em} \underline V_i^2 \le e_i^2+f_i^2 \le \overline V_i^2    & i& \in \mathcal{B} \label{voltageAtBus} \\
  & \hspace{0.25em}  p_i^{\text{min}}  \le p_i^g \le p_i^{\text{max}}     & i& \in \mathcal{G} \label{activeAtGenerator} \\
  & \hspace{0.25em} q_i^{\text{min}}  \le q_i^g \le q_i^{\text{max}}     & i& \in \mathcal{G}. \label{reactiveAtGenerator}
\end{align}
\end{subequations}
Here, the objective function $C_i(p_i^g)$ is typically linear or convex quadratic in the real power output $p_i^g$ of generator $i$. Constraints \eqref{activeAtBus} and \eqref{reactiveAtBus} correspond to the conservation of active and reactive power flows at each bus, respectively. $\delta(i)$ denotes the set of neighbor buses of bus $i$. Constraint \eqref{voltageAtBus} restricts voltage magnitude at each bus. As noted above, $\underline V_i$ and $\overline V_i$ are both close to  1 p.u. at each bus $i$. Constraints \eqref{activeAtGenerator} and \eqref{reactiveAtGenerator}, respectively, limit the active and reactive power output of each generator to respect its physical capability. 

Note that the rectangular formulation \eqref{rect} is a nonconvex quadratic optimization problem. However, quite importantly, we can observe that all the nonlinearity and nonconvexity comes from one of the following three forms:
(1) $e_i^2+f_i^2=|V_i|^2$, (2) $e_ie_j+f_if_j=|V_i||V_j|\cos(\theta_i-\theta_j)$, (3) $e_if_j-f_ie_j=-|V_i||V_j|\sin(\theta_i-\theta_j)$. To capture this nonlinearity, we define new variables $c_{ii}$, $c_{ij}$ and $s_{ij}$ for each bus $i$ and each transmission line $(i,j)$ as
$c_{ii} = e_i^2 + f_i^2$, $c_{ij} = e_ie_j+f_if_j$, $s_{ij}=e_if_j-e_jf_i$. These new variables satisfy the relation $c_{ij}^2 + s_{ij}^2=c_{ii}c_{jj}$. With a change of variables, we can introduce an alternative formulation of the OPF problem as follows:
\begin{subequations} \label{SOCP}
\begin{align}
 \min  &\hspace{0.5em}  \sum_{i \in \mathcal{G}} C_i(p_i^g) \\
  \mathrm{s.t.}   &\hspace{0.5em} p_i^g-p_i^d = G_{ii}c_{ii} + \sum_{j \in \delta(i)}\left( G_{ij}c_{ij} -B_{ij}s_{ij}\right)   & i& \in \mathcal{B} \label{activeAtBusR} \\
  & \hspace{0.5em} q_i^g-q_i^d = -B_{ii}c_{ii} + \sum_{j \in \delta(i)}\left( -B_{ij}c_{ij} -G_{ij}s_{ij}\right)  & i& \in \mathcal{B} \label{reactiveAtBusR} \\
  & \hspace{0.5em} \underline V_i^2 \le c_{ii} \le \overline V_i^2    & i& \in \mathcal{B}  \label{voltageAtBusR} \\
  & \hspace{0.5em} c_{ij}=c_{ji}, \ \ s_{ij}=-s_{ji}    &(&i,j) \in \mathcal{L} \label{cosine_sine}\\
  & \hspace{0.5em}  c_{ij}^2+s_{ij}^2  = c_{ii}c_{jj}     &(&i,j) \in \mathcal{L} \label{coupling}\\
\notag  & \hspace{0.5em} \eqref{activeAtGenerator},   \eqref{reactiveAtGenerator}. 
\end{align}
\end{subequations}

This formulation was first introduced in \cite{Exposito99} and \cite{Jabr06}. It is an exact formulation for OPF on a radial network in the sense that the optimal power output of \eqref{SOCP} is also optimal for \eqref{rect} and one can always recover the voltage phase angles $\theta_i$'s by solving the following system of linear equations with the optimal solution $c_{ij}, s_{ij}$:
\begin{equation} 
\theta_j - \theta_i = \atan2({s_{ij}},{c_{ij}}) \quad (i,j) \in \mathcal{L},  \label{arctan cons}
\end{equation}
which then provide an optimal voltage phasor solution to \eqref{rect} (see e.g. \cite{Zhang12}). Here, in order to cover the entire range of $2\pi$, we use the $\atan2(y,x)$ function\footnotemark\footnotetext{$\operatorname{atan2}(y, x) = \begin{cases}
\arctan\frac y x & \quad x > 0 \\
\arctan\frac y x + \pi& \quad y \ge 0 , x < 0 \\
\arctan\frac y x - \pi& \quad y < 0 , x < 0 \\
+\frac{\pi}{2} & \quad y > 0 , x = 0 \\
-\frac{\pi}{2} & \quad y < 0 , x = 0 \\
\text{undefined} & \quad y = 0, x = 0
\end{cases}$}, which takes value in $(-\pi, \pi]$, rather than $[-\pi/2, \pi/2]$ as is the case of the regular arctangent function. Unfortunately, for meshed networks, the above formulation \eqref{SOCP} can be a strict relaxation of the OPF problem. The reason is that,  given an optimal solution $c_{ij}, s_{ij}$ for all edges $(i,j)$ of \eqref{SOCP}, it does not guarantee that $\atan2({s_{ij}},{c_{ij}})$ sums to zero over all cycles. In other words, the optimal solution of \eqref{SOCP} may not be feasible for the original OPF problem \eqref{rect}. This issue can be fixed by directly imposing (\ref{arctan cons}) as a constraint \citep{Jabr07, Jabr08}. Thus \eqref{SOCP} together with (\ref{arctan cons}) is a valid formulation for OPF in mesh networks. Note that the constraints involving the $\atan2$ function are nonconvex. 

\noindent{\bf Line Flow Constraints:}
Typically, the OPF problem also involves the so-called transmission constraints  on transmission lines. In the literature, different types of line flow limits are suggested. We list a few of them below \citep{madani2015, coffrin2015} and present how they can be expressed in the space of $(c,s,\theta)$ variables:
\begin{enumerate}
\item
Real power flow on line $(i,j)$: $-G_{ij}c_{ii} + G_{ij}c_{ij} -B_{ij}s_{ij}$
\item
Squared voltage difference magnitude on line $(i,j)$: $c_{ii} - 2c_{ij} + c_{jj}$
\item
Squared current magnitude on line $(i,j)$: $(B_{ij}^2+G_{ij}^2)(c_{ii} - 2c_{ij} + c_{jj})$
\item
Squared apparent power flow on line $(i,j)$: $[-G_{ij}c_{ii} + G_{ij}c_{ij} -B_{ij}s_{ij}]^2+ [B_{ij}c_{ii} -  B_{ij}c_{ij} - G_{ij}s_{ij}]^2  $
\item
Angle difference on line $(i,j)$: $\theta_i - \theta_j$ or $\atan2({s_{ij}},{c_{ij}})$, see equation \eqref{angle diff theta cs} for details
\end{enumerate}
We omit such constraints for the brevity of discussion. However, last two constraints are included in our computational experiments whenever explicit bounds are given.

\section{Comparison of Convex Relaxations}
\label{Comparison Radial}

In this section, we first present six different convex relaxations of the OPF problem. In particular, we consider the McCormick, SOCP, and SDP relaxations of both the rectangular formulation \eqref{rect} and the alternative formulation \eqref{SOCP}. Then, we analyze their relative strength by comparing their feasible regions. This comparison is an important motivator for the approach we take in the rest of the paper to generate strong SOCP relaxations. We discuss this in Section \ref{sec:choose}.

\subsection{Standard Convex Relaxations}

\subsubsection{McCormick Relaxation of Rectangular Formulation ($\mathcal{R}_{M}$).} 
As shown in \cite{mccormick}, the convex hull of the set $\{(x,y,w) : w=xy, \; (x,y) \in [\underline x, \overline x]\times[\underline y, \overline y]\}$ is given by
\begin{equation*}
\left\{ (x,y,w) :  \max\{\underline y x + \underline x y - \underline x \underline y, \overline y x + \overline x y - \overline x \overline y \} \le w \le \min\{\underline y x + \overline x y - \overline x \underline y, \overline y x + \underline x y - \underline x \overline y \} \right\},
\end{equation*}
which we denote as $M(w=xy)$. We use this result to construct McCormick envelopes for the quadratic terms in the rectangular formulation \eqref{rect}. In particular, let us first define the following new variables for each edge $(i,j)\in\mathcal{L}$: $E_{ij}=e_ie_j$, 
$F_{ij}=f_if_j$, 
$H_{ij}=e_if_j$, and for each bus $i\in\mathcal{B}$: $E_{ii}=e_i^2$, $F_{ii}=f_i^2$. Consider the following set of constraints:
%
\begin{subequations} \label{rect McCormick}
\begin{align}
&\hspace{0.5em} p_i^g-p_i^d = G_{ii}(E_{ii}+F_{ii}) + \sum_{j \in \delta(i)}[ G_{ij}(E_{ij}+F_{ij}) -B_{ij}(H_{ij}-H_{ji})]   & i& \in \mathcal{B}  \label{eq:RMp} \\
  & \hspace{0.5em} q_i^g-q_i^d = -B_{ii}(E_{ii}+F_{ii}) + \sum_{j \in \delta(i)}[ -B_{ij}(E_{ij}+F_{ij}) -G_{ij}(H_{ij}-H_{ji})]  & i& \in \mathcal{B} \label{eq:RMq}\\
  & \hspace{0.5em} \underline V_i ^2 \le E_{ii}+F_{ii} \le \overline V_i^2    & i& \in \mathcal{B} \label{eq:E+Fbounds} \\
  & \hspace{0.5em} -\overline V_i   \le e_i, f_i \le \overline V_i     & i& \in \mathcal{B}  \label{eq:efbounds} \\
  & \hspace{0.5em}   M(E_{ij}=e_i e_j), M(F_{ij}=f_i f_j) , M(H_{ij}=e_i f_j) &(&i,j) \in \mathcal{L} \label{eq:MEef1}\\
  & \hspace{0.5em}   M(E_{ii}=e_i^2) , M(F_{ii}=f_i^2), \; E_{ii}, F_{ii} \ge 0  & i& \in \mathcal{B}, \label{eq:MEef2}
\end{align}
\end{subequations}
where the McCormick envelops in \eqref{eq:MEef1}-\eqref{eq:MEef2} are constructed using the bounds given in \eqref{eq:efbounds}. Using these McCormick envelopes, we obtain a convex relaxation of the rectangular formulation \eqref{rect} with the feasible region denoted as 
\begin{align}\label{eq:RM}
\mathcal{R}_M = \{ (p, q, e, f, E, F, H) : (\ref{rect McCormick}), \eqref{activeAtGenerator}, \eqref{reactiveAtGenerator} \}.
\end{align} 
Note that this feasible region is a polyhedron. If the objective function $C_i(p_i^g)$ is linear, then we have a linear programming relaxation of the OPF problem.

\subsubsection{McCormick Relaxation of Alternative Formulation ($\mathcal{A}_{M}$).} Using the similar technique on the alternative formulation \eqref{SOCP}, we define new variables
$C_{ij} = c_{ij}^2$, $S_{ij} = s_{ij}^2$, $D_{ij} = c_{ii}c_{jj}$ for 
each edge $(i,j)\in\mathcal{L}$, and consider the following set of constraints:
\begin{subequations} \label{alt McCormick}
\begin{align}
  & \hspace{0.5em}  C_{ij}+S_{ij}  = D_{ij}     &(&i,j) \in \mathcal{L} \label{equal socp} \\
  & \hspace{0.5em} -\overline V_i\overline V_j \le  c_{ij}, s_{ij} \le \overline V_i\overline V_j     &(&i,j) \in \mathcal{L} \label{alt mccormickbnd}\\   
  & \hspace{0.5em}   M(C_{ij}=c_{ij}^2), M(S_{ij}=s_{ij}^2), M(D_{ij}=c_{ii}c_{jj}), C_{ij}, S_{ij} \ge 0 &(&i,j) \in \mathcal{L}, \label{alt mccormick cone}
\end{align}
\end{subequations}
where the McCormick envelops in \eqref{alt mccormick cone} are constructed using the bounds given in \eqref{alt mccormickbnd} and \eqref{voltageAtBusR}. 
Denote the feasible region of the corresponding convex relaxation as 
\begin{align} \label{eq:AM}
\mathcal{A}_M = \{ (p, q, c, s, C, S, D) : (\ref{alt McCormick}), \eqref{activeAtBusR}- \eqref{cosine_sine}, \eqref{activeAtGenerator},   \eqref{reactiveAtGenerator}\}.
\end{align} 
Again, $\mathcal{A}_M$ is a polyhedron.

\subsubsection{SDP Relaxations of Rectangular Formulation ($\mathcal{R}_{SDP}, \mathcal{R}^c_{SDP}, \mathcal{R}^r_{SDP}$).} To apply SDP relaxation to the rectangular formulation \eqref{rect}, define a hermitian matrix $X\in\mathbb{C}^{|\mathcal{B}|\times|\mathcal{B}|}$, i.e., $X=X^*$, where $X^*$ is the conjugate transpose of $X$. Consider the following set of constraints:
\begin{subequations} \label{rect sdp}
\begin{align}
 &\hspace{0.5em} p_i^g-p_i^d = G_{ii}X_{ii} + \sum_{j \in \delta(i)}[ G_{ij} \Re(X_{ij}) +B_{ij} \Im(X_{ij})]   & i& \in \mathcal{B} \label{rect sdp first}  \\
  & \hspace{0.5em} q_i^g-q_i^d = -B_{ii}X_{ii} + \sum_{j \in \delta(i)}[ -B_{ij}\Re(X_{ij}) +G_{ij}\Im(X_{ij})]  & i& \in \mathcal{B} \label{rect sdp second} \\
  & \hspace{0.5em}  \underline V_i^2 \le X_{ii} \le \overline V_i^2   & i& \in \mathcal{B}\label{rect sdp last} \\
  & \hspace{0.5em} X \textup{ is hermitian,} \ X \succeq  0,   \label{psd cons} 
\end{align}
\end{subequations}
where $\Re(x)$ and $\Im(x)$ are the real and imaginary parts of the complex number $x$, respectively. Let $V$ denote the vector of voltage phasors with the $i$-th entry $V_i =e_i+\mathrm{i}f_i$ for each bus $i\in\mathcal{B}$. If $X=VV^*$, then rank$(\Re(X))$ and rank$(\Im(X))$ are both equal to 2, and \eqref{rect sdp first}-\eqref{rect sdp last} exactly recovers \eqref{activeAtBus}-\eqref{voltageAtBus}. By ignoring the rank constraints, we come to the standard SDP relaxation of the rectangular formulation \eqref{rect}, whose feasible region is defined as 
\begin{align}\label{eq:RSDPcomplex}
\mathcal{R}^c_{SDP} = \{(p, q, W) : \eqref{rect sdp}, \eqref{activeAtGenerator},   \eqref{reactiveAtGenerator} \}.
\end{align}
This SDP relaxation is in the complex domain. There is also an SDP relaxation in the real domain by defining $W\in\mathbb{R}^{2|C|\times 2|C|}$. The associated constraints are
 \begin{subequations}\label{rect sdp real}
 \begin{align}
	 &\hspace{0.5em} p_i^g-p_i^d = G_{ii}(W_{ii}+W_{i'i'}) + \sum_{j \in \delta(i)}[ G_{ij} (W_{ij} + W_{i'j'})  - B_{ij} (W_{ij'}-W_{ji'})]   & i& \in \mathcal{B} \label{rect sdp real first}  \\
	 & \hspace{0.5em} q_i^g-q_i^d = -B_{ii}(W_{ii}+W_{i'i'}) + \sum_{j \in \delta(i)}[ -B_{ij}(W_{ij} + W_{i'j'}) - G_{ij}(W_{ij'}-W_{ji'})]  & i& \in \mathcal{B} \label{rect sdp real second} \\
	 & \hspace{0.5em}  \underline V_i^2 \le W_{ii} + W_{i'i'}\le \overline V_i^2   & i& \in \mathcal{B}\label{rect sdp real last} \\
	 & \hspace{0.5em} W \succeq  0,   \label{real psd cons} 
 \end{align}
 \end{subequations}
where $i'=i+|\mathcal{B}|$ and $j'=j+|\mathcal{B}|$. If $W=[e; f][e^T, f^T]$, i.e. rank($W$)=1, then \eqref{rect sdp real first}-\eqref{rect sdp real last} exactly recovers \eqref{activeAtBus}-\eqref{voltageAtBus}. We denote the feasible region of this SDP relaxation in the real domain as 
\begin{align}\label{eq:RSDPreal}
R_{SDP}^{r}=\{(p,q,W) : \eqref{rect sdp real first}-\eqref{real psd cons}, \eqref{activeAtGenerator}-\eqref{reactiveAtGenerator}\}.
\end{align}
The SDP relaxation in the real domain is first proposed in \cite{Bai2008, Bai2009}, and \cite{Lavaei12}. The SDP relaxation in the complex domain is formulated in \cite{bose2011optimal} and \cite{Zhang2013} and is widely used in the literature now for its notational simplicity. These two SDP relaxations produce the same bound, since the solution of one can be used to derive a solution to the other with the same objective function value. See e.g., Section 3.3 in \cite{Taylorbook} for a formal proof.  Henceforth, we refer to the SDP relaxation as $\mathcal{R}_{SDP}$ that is $\mathcal{R}_{SDP} := \mathcal{R}^c_{SDP} = \mathcal{R}^r_{SDP}$, 
where the second equality (with some abuse of notation) is meant to imply that the two relaxations have the same projection in the space of the $p, q$ variables.


\subsubsection{SOCP Relaxation of Rectangular Formulation ($\mathcal{R}_{SOCP}$).}\label{sec:SOCP rect} We can further apply SOCP relaxation to the SDP constraint (\ref{psd cons}) by only imposing the following constraints on all the $2\times2$ submatrices of $X$ for each line $(i,j) \in \mathcal{L}$,
\begin{equation}\label{rect socp cone}
\begin{bmatrix}
X_{ii} & X_{ij} \\
X_{ji} & X_{jj}
\end{bmatrix} \succeq 0 \quad (i,j) \in \mathcal{L}. 
\end{equation}
This is a relaxation of \eqref{psd cons}, because \eqref{psd cons} requires all principal submatrices of $X$ be SDP (see e.g., \cite{Horn}). Note that \eqref{rect socp cone} has a $2\times 2$ hermitian matrix, i.e., $X_{ij}=X_{ji}^*$. Using the Sylvester criterion, \eqref{rect socp cone} is equivalent to $X_{ij}X_{ji}\leq X_{ii}X_{jj}$ and $X_{ii}, X_{jj}\geq 0$. The first inequality can be written as $\Re(X_{ij})^2 + \Im(X_{ij})^2+ \left(\frac{X_{ii}-X_{jj}}{2}\right)^2\leq \left(\frac{X_{ii}+X_{jj}}{2}\right)^2$, which is an SOCP constraint in the real domain. Thus, we have an SOCP relaxation of the rectangular formulation with the feasible region defined as 
\begin{align}\label{eq:RSOCP}
\mathcal{R}_{SOCP}= \{ (p, q, X) : \eqref{activeAtGenerator},   \eqref{reactiveAtGenerator}, (\ref{rect sdp first})-(\ref{rect sdp last}), (\ref{rect socp cone}) \}.
\end{align}
This SOCP relaxation is also proposed in the literature, see e.g., \cite{Madani}. 
In \cite{low2014convexi}, this relaxation is proven to be equivalent to the SOCP relaxation of DistFlow model proposed in \cite{baran1989_, baran1989}.

\subsubsection{SOCP Relaxation of Alternative Formulation ($\mathcal{A}^{*}_{SOCP}$).} The nonconvex coupling constraints \eqref{coupling} in the alternative formulation \eqref{SOCP} can be relaxed 
as follows,
\begin{align}
c_{ij}^2+s_{ij}^2  \le c_{ii}c_{jj}     \quad (i,j) \in \mathcal{L}.  \label{SOCP cone}
\end{align}
It is easy to see that \eqref{SOCP cone} can be rewritten as $c_{ij}^2+s_{ij}^2+ \left(\frac{c_{ii}-c_{jj}}{2}\right)^2 \leq \left(\frac{c_{ii}+c_{jj}}{2}\right)^2$, which represents a rotated SOCP cone in $\mathbb{R}^4$. Note that the SOCP cone \eqref{SOCP cone} is the convex hull of \eqref{coupling}. 
Using \eqref{SOCP cone}, we have an SOCP relaxation of the alternative formulation \eqref{SOCP}.
Denote the feasible region of this SOCP relaxation as 
\begin{align}\label{eq:ASOCP}
{\mathcal{A}^*_{SOCP}} = \{ (p, q, c, s) :  (\ref{SOCP cone}) ,   \eqref{activeAtBusR}- \eqref{cosine_sine},  \eqref{activeAtGenerator},   \eqref{reactiveAtGenerator}\}.
\end{align}
This is the classic SOCP relaxation first proposed in \cite{Jabr06}. 

\subsubsection{SDP Relaxation of Alternative Formulation ($\mathcal{A}_{SDP}$).}
We can also apply SDP relaxation to the nonconvex quadratic constraints \eqref{coupling} in the alternative formulation \eqref{SOCP}. In particular, define a vector  $z \in \mathbb{R}^{2|\mathcal{L}|+|\mathcal{B}|}$, of which the first ${2|\mathcal{L}|}$ components are indexed by the set of branches $(i,j)\in\mathcal{L}$, and the last ${|\mathcal{B}|}$ components are indexed by the set of buses $i\in\mathcal{B}$. Essentially, $z$ represents $((c_{ij})_{(i,j)\in\mathcal{L}}, (s_{ij})_{(i,j)\in\mathcal{L}}, (c_{ii})_{i\in\mathcal{B}})$. Then define a real matrix variable $Z$ to approximate $zz^T$ and consider the following set of constraints:
\begin{subequations} \label{alt sdp}
\begin{align}
  & \hspace{0.5em}  Z_{(ij),(ij)}+Z_{(i'j'),(i'j')} =  Z_{(ii),(jj)}    &(&i,j) \in \mathcal{L} \label{alt sdp quad} \\
  & \hspace{0.5em} Z \succeq  zz^T  \label{psd cone} \\
  & \hspace{0.5em} Z_{(ij),(ij)} \le (\overline V_i\overline V_j)^2, \; Z_{(i'j'),(i'j')}  \le (\overline V_i\overline V_j)^2 &(&i,j) \in \mathcal{L}  \label{diagonal line} \\
  & \hspace{0.5em} Z_{(ii),(ii)}   \le (\underline V_i^2 +  \overline V_i^2 )c_{ii} - (\underline V_i  \overline V_i )^2  &i& \in \mathcal{B},  \label{diagonal bus}
\end{align}
\end{subequations}
where $i'=i+|\mathcal{L}|$ and $j'=j+|\mathcal{L}|$.
Constraints \eqref{alt sdp quad} and \eqref{psd cone} are the usual SDP relaxation of \eqref{coupling}, and constraints (\ref{diagonal line}) and (\ref{diagonal bus}) are used to properly upper bound the diagonal elements of $Z$, where constraint \eqref{diagonal bus} follows from applying McCormick envelopes on the squared terms $c_{ii}^2$. In particular, if $z_{(ij)}$ is restricted to be between $[\underline{z}_{(ij)}, \overline{z}_{(ij)}]$, then the McCormick upper envelope for the diagonal element $Z_{(ij),(ij)}$ is given as $Z_{(ij),(ij)} \le (\underline z_{(ij)} + \overline z_{(ij)}) z_{(ij)} - \underline z_{(ij)}\overline z_{(ij)}.$
Denote the feasible region of this SDP relaxation of the alternative formulation \eqref{SOCP} as  
\begin{align} \label{eq:ASDP}
\mathcal{A}_{SDP} = \{ (p, q, c, s, Z) : (\ref{alt sdp}), \eqref{activeAtBusR}-\eqref{cosine_sine} ,  \eqref{activeAtGenerator},   \eqref{reactiveAtGenerator}\}.
\end{align} 
Note that this SDP relaxation of the alternative formulation is derived using standard techniques, but to the best of our knowledge, it has not been discussed in the literature.

\subsection{Comparison of Relaxations}
The main result of this section is Theorem \ref{all relaxations}, which presents the relative strength of the various convex relaxations introduced above. In order to compare relaxations, they must be in the same variable space. Also the objective function depends only on the value of the real powers. Therefore, we study the feasible regions of the above convex relaxations projected to the space of real and reactive powers, i.e. the $(p,q)$ space. We use `$\hat{\textup{ }}$' to denote this projection. For example, $\hat{\mathcal{R}}_{M}$ is the projection of $\mathcal{R}_{M}$ to the $(p,q)$ space.

\begin{thm}  \label{all relaxations} Let  $\mathcal{R}_{M}$, $\mathcal{A}_{M}$, $\mathcal{R}_{SDP}$, $\mathcal{R}_{SOCP}$, $\mathcal{A}^{*}_{SOCP}$, and ${\mathcal{A}}_{SDP}$ be the McCormick relaxation of the rectangular formulation \eqref{rect}, the McCormick relaxation of the alternate formulation \eqref{SOCP}, the SDP relaxation of the rectangular formulation \eqref{eq:RSDPcomplex}, the SOCP relaxation of the rectangular formulation \eqref{eq:RSOCP}, the SOCP relaxation of the alternative formulation \eqref{eq:ASOCP}, and the SDP relaxation of the alternative formulation  \eqref{eq:ASDP}, respectively. Then:
\begin{enumerate}
\item The following relationships hold between the feasible regions of the convex relaxations when projected onto the $(p,q)$ space:
\begin{align}\label{eq:inclusion}
\begin{split}
\hat{\mathcal{R}}_{SDP}  \subseteq \hat{\mathcal{R}}_{SOCP} & = \hat{\mathcal{A}}^*_{SOCP}  \subseteq \hat{\mathcal{A}}_{SDP}      \\
 &\quad \quad  \ \ \rotatebox{90}{$\supseteq$}    \\
  \hat{\mathcal{R}}_{M} \ \ &   \supseteq \quad \hat{\mathcal{A}}_{M}
\end{split}
\end{align}
\item All the inclusions in \eqref{eq:inclusion} can be strict.
\end{enumerate}
\end{thm}

In Section \ref{sec:containments} we present the proof of part (1) of Theorem \ref{all relaxations}, and in Section \ref{sec:propercontainment} we present examples to verify part (2) of the theorem. 

\subsubsection{Pairwise comparison of relaxations}\label{sec:containments}
The proof of part (1) of Theorem \ref{all relaxations} is divided into Propositions \ref{R_M A_M}-\ref{A_SOCP A_SDP}. 


\begin{prop} \label{R_M A_M}
$\hat{\mathcal{R}}_M \supseteq \hat{\mathcal{A}}_M $.
\end{prop}
\begin{proof}
In order to prove this result, we want to show that for any given $(p,q,c,s,C,S,D)\in\mathcal{A}_M$, we can find $(e,f,E,F,H)$ such that $(p,q,e,f,E,F,H)$ is in $\mathcal{R}_M$. For this purpose, set $e_i=f_i=0$ for $i \in \mathcal{B}$, $E_{ij}=c_{ij}$, $F_{ij}$=0, $H_{ij}=s_{ij}$, $H_{ji}=0$ for each $(i,j)\in\mathcal{L}$, and $E_{ii}=c_{ii}$ and $F_{ii}=0$ for each $i\in\mathcal{B}$. By this construction, we have $E_{ii} + F_{ii} = c_{ii}$, $E_{ij} + F_{ij} = c_{ij}$, and $H_{ij} - H_{ji} = s_{ij}$. Therefore, \eqref{activeAtBusR}-\eqref{reactiveAtBusR} implies that the constructed $E,F,H$ satisfy \eqref{eq:RMp}-\eqref{eq:RMq}; \eqref{voltageAtBusR} implies \eqref{eq:E+Fbounds}; \eqref{eq:efbounds} is trivially satisfied since $e_i=f_i=0$. Now to see the McCormick envelopes \eqref{eq:MEef1}-\eqref{eq:MEef2} are satisfied, consider $M(E_{ij}=e_ie_j)$. Using the bounds \eqref{eq:efbounds} on $e_i,f_i$,  $M(E_{ij}=e_ie_j)$ can be written as
\begin{align*}
\max\{-\overline{V}_ie_j-\overline{V}_je_i-\overline{V}_i\overline{V}_j,\; \overline{V}_ie_j +\overline{V}_je_i-\overline{V}_i\overline{V}_j\} \leq E_{ij} \\
E_{ij} \leq\min\{-\overline{V}_ie_j+\overline{V}_je_i+\overline{V}_i\overline{V}_j,\; \overline{V}_ie_j -\overline{V}_je_i+\overline{V}_i\overline{V}_j\}.		
\end{align*}
Since $e_i=0$ for all $i\in\mathcal{B}$, the above inequalities reduce to $-\overline{V}_i\overline{V}_j\leq E_{ij}\leq \overline{V}_i\overline{V}_j$, which is implied by $E_{ij}=c_{ij}$ and the bounds \eqref{alt mccormickbnd}. Similar reasoning can be applied to verify that the other McCormick envelopes in \eqref{eq:MEef1}-\eqref{eq:MEef2} are satisfied. Finally, it is straightforward to see that $E_{ii}  = c_{ii} \geq 0, F_{ii} = 0$. Therefore, the constructed $(p,q,e,f,E,F,H)$ is in $\mathcal{R}_M$. 
%
\end{proof}
In fact, the above argument only relies on constraints \eqref{activeAtBusR}-\eqref{voltageAtBusR} and \eqref{alt mccormickbnd} in $\mathcal{A}_M$. This suggests that $\hat{\mathcal{A}}_M$ may be strictly contained in $\hat{\mathcal{R}}_M$, which is indeed the case shown in Section \ref{sec:propercontainment}. 

\begin{prop} \label{A_M A_S} 
$\hat{\mathcal{A}}_M \supseteq \hat{\mathcal{A}}^*_{SOCP}$.
\end{prop}
\begin{proof}
It suffices to prove that $\text{proj}_{p, q, c, s} \mathcal{A}_M \supseteq \mathcal{A}^*_{SOCP}$. For this purpose, we want to show that for each $(p, q, c, s) \in \mathcal{A}^*_{SOCP} $, there exists $(C,S,D)$ such that $(p, q, c, s, C,S,D) \in \mathcal{A}_M$. In particular, set $C_{ij} = c_{ij}^2 $, $S_{ij} = c_{ii}c_{jj} -c_{ij}^2 \geq s_{ij}^2$, and 
$D_{ij} = c_{ii}c_{jj}$  for each $(i,j) \in \mathcal{L}$.
Note that $C_{ij}$ and $D_{ij}$ satisfy constraints (\ref{alt mccormick cone}) by the definition of McCormick envelopes. So, it remains to verify if $S_{ij}$ satisfies $M(S_{ij}=s_{ij}^2)$. 
This involves verifying:
\begin{align}\label{last ineq}
 2(\overline V_j \overline V_i)|s_{ij}| - (\overline V_j \overline V_i)^2 \le  S_{ij} \le   (\overline V_j \overline V_i)^2.
\end{align}
The first inequality holds because $S_{ij}\geq s_{ij}^2$ and $s_{ij}^2-2(\overline V_j \overline V_i)|s_{ij}|+(\overline V_j \overline V_i)^2=(|s_{ij}|-\overline V_j \overline V_i)^2\geq 0$. The second inequality holds since
$
S_{ij} =  c_{ii}c_{jj} -c_{ij}^2 \le  c_{ii}c_{jj} \le (\overline V_i \overline V_j)^2.
$
Thus, the result follows.
\end{proof}

The next result is straightforward, however we present it for completeness.
\begin{prop}  \label{A_SOCP R_SOCP}
$\hat{\mathcal{A}}^*_{SOCP}=\hat{\mathcal{R}}_{SOCP}$.
\end{prop}
\begin{proof}
Note that $X_{ii} = c_{ii}$, $X_{jj} = c_{jj}$ and $X_{ij} = c_{ij} + \mathrm{i}s_{ij}$ forms a bijection between the $(c,s)$ variables in $\mathcal{A}^*_{SOCP}$ and the $X$ variable in $\mathcal{R}_{SOCP}$. Given this bijection, \eqref{activeAtBus}-\eqref{voltageAtBus} are equivalent to \eqref{rect sdp first}-\eqref{rect sdp last} and \eqref{cosine_sine} is equivalent to $X_{ij}=X_{ji}^*$. As mentioned in Section \ref{sec:SOCP rect}, the hermitian condition in \eqref{rect socp cone} is equivalent to $|X_{ij}|^2\leq X_{ii}X_{jj}$ and $X_{ii}, X_{jj}\geq 0$, which is exactly \eqref{SOCP cone}. 
\end{proof}


\begin{prop}  \label{A_SOCP A_SDP}
$\hat{\mathcal{A}}^*_{SOCP}  \subseteq \hat{\mathcal{A}}_{SDP}$. 
\end{prop}
\begin{proof}
It suffices to show that $\mathcal{A}^*_{SOCP}  \subseteq \text{proj}_{p, q, c, s} \mathcal{A}_{SDP}$. For this, we can show that given any $(p, q, c, s) \in \mathcal{A}^*_{SOCP}$, there exists a $Z$ such that $(p, q, c, s, Z) \in \mathcal{A}_{SDP}$. In particular, we construct $Z = zz^T + Z'$, where $z=((c_{ij})_{(i,j)\in\mathcal{L}}, (s_{ij})_{(i,j)\in\mathcal{L}},(c_{ii})_{i\in\mathcal{B}})$, and $Z'$ is a diagonal matrix with the first $|\mathcal{L}|$ entries on the diagonal equal to $c_{ii}c_{jj}-(c_{ij}^2+s_{ij}^2)$ for each $(i,j)\in\mathcal{L}$ and all other entries equal to zero. By this construction, \eqref{alt sdp quad} is satisfied. Also $Z\succeq zz^T$, i.e. \eqref{psd cone} is satisfied. Note that the first $|\mathcal{L}|$ entries on the diagonal of $Z$ is equal to $c_{ii}c_{jj}-s_{ij}^2$. Therefore, the bounds in \eqref{diagonal line} are equivalent to $c_{ii}c_{jj}-s_{ij}^2\leq \overline{V}_i^2\overline{V}_j^2$, $s_{ij}^2\leq \overline{V}_i^2\overline{V}_j^2$, where the first one follows from \eqref{voltageAtBusR} and the second one follows from \eqref{voltageAtBusR} and \eqref{SOCP cone}.
Finally, constraint \eqref{diagonal bus} is the McCormick envelope of $c_{ii}^2$ with the same bounds in \eqref{voltageAtBusR}, therefore it is also satisfied by the solution of $\mathcal{A}^*_{SOCP}$.  
 \end{proof}


%

\subsubsection{Strictness of Inclusions.}\label{sec:propercontainment}
Now, we prove by examples that all the inclusions in Theorem  \ref{all relaxations} can be strict.
\begin{enumerate}
\item $\hat{\mathcal{A}}^*_{SOCP} \subset \hat{\mathcal{A}}_{M} \subset  \hat{\mathcal{R}}_{M}$: Consider the 9-bus radial network with quadratic objective function from   \cite{kocuk2014}. We compare the strength of these three types of relaxations in Table~\ref{McCormick results}, which shows that the inclusions can be strict.
\begin{table}[H] 
\captionsetup{width=.78\textwidth}\caption{Percentage optimality gap for $\mathcal{R}_M$, $\mathcal{A}_M$, and $\mathcal{A}^*_{SOCP}$ with respect to global optimality found by BARON.} 
\label{McCormick results}
\begin{center}
\begin{tabular}{ccrrr}
\hline
    case       &  objective          &  $\mathcal{R}_M$ &  $\mathcal{A}_M$  &       $\mathcal{A}^*_{SOCP}$ \\
\hline
     9-bus &  quadratic &      89.46 &      69.13 &      52.69  \\
\hline
\end{tabular}  
\end{center}
\end{table}
{Here, the percentage optimality gap  is calculated as $ 100 \times \frac{z^{\text{*}} - z^{\text{LB}}}{z^{\text{*}}}$,
where $z^{\text{*}}$ is the global optimal objective function value found by BARON \citep{BARON}  and $z^{\text{LB}}$ is the optimal objective cost of a particular relaxation.
}

\item $\hat{\mathcal{A}}^*_{SOCP} \subset \hat{\mathcal{A}}_{SDP}$: Consider a 2-bus system. Since there is only one transmission line, $\mathcal{A}^*_{SOCP}$ has only one SOCP constraint	
\begin{equation} \label{2-bus socp}
c_{12}^2+s_{12}^2 \le c_{11} c_{22},
\end{equation}
while the SDP relaxation $\mathcal{A}_{SDP}$ has
\begin{equation}
Z \succeq zz^T \iff 
\begin{bmatrix} 
1 & c_{12} & s_{12} & c_{11} & c_{22} \\
 c_{12}  & Z_{11} & Z_{12} & Z_{13} & Z_{14} \\
 s_{12}  & Z_{21} & Z_{22} & Z_{23} & Z_{24} \\
 c_{11}  & Z_{31} & Z_{32} & Z_{33} & Z_{34} \\
 c_{22}  & Z_{41} & Z_{42} & Z_{43} & Z_{44}
\end{bmatrix}
\succeq 0
\end{equation}
with the additional constraints  $ Z_{11}+ Z_{22}= Z_{43}$, $Z_{11} \le 1.4641$, $Z_{22} \le 1.4641$, $Z_{33} \le 2.02c_{12}-0.9801$ and $Z_{44} \le 2.02s_{12}-0.9801$, assuming that $\underline V_1=\underline V_2 = 0.9$ and  $\overline V_1=\overline V_2 = 1.1$. 


Now, consider a point $(c_{12}, s_{12}, c_{11}, c_{22}) = (1.000,0.100,1.000,1.000)$, which clearly violates constraint (\ref{2-bus socp}). However, one can extend this point in SDP relaxation as follows:
\begin{equation}
\begin{bmatrix} 
1 & z^T \\
z & Z
\end{bmatrix}
=
\begin{bmatrix} 
1.000&  1.000 &  0.100   &1.000  & 1.000 \\
1.000&   1.006 &  0.100  &  0.997 &  0.997\\
0.100  & 0.100  &  0.017 &  0.100 &  0.100\\
1.000 &  0.997  & 0.100&   1.029 &  1.023\\
1.000 &  0.997&   0.100  & 1.023 &  1.029
\end{bmatrix}
\succeq 0
\end{equation}
This proves our claim that the SDP relaxation $\hat{\mathcal{A}}_{SDP}$  can be weaker than the SOCP relaxation $\hat{\mathcal{A}}^*_{SOCP}$. 

\item $\mathcal{R}_{SDP} \subset \mathcal{R}_{SOCP}$: Although this relation holds as equality for radial networks \citep{Sojoudi}, the inclusion can be strict for meshed networks. For example, Table \ref{all socp socpa} demonstrates this fact, e.g.,  the instance case6ww from the MATPOWER library \citep{Matpower}.

\end{enumerate}

\subsection{Our choice of convex relaxation}\label{sec:choose}

We discuss some consequences of Theorem \ref{all relaxations} here.  

Among LPs, SOCPs, and SDPs, the most tractable relaxation are LP relaxations. In a recent paper \cite{kocuk2014}, we showed how $\mathcal{A}_{M}$ may be used (together with specialized cutting planes) to solve tree instances of OPF globally. Theorem \ref{all relaxations} provides a theoretical basis for selection of $\mathcal{A}_{M}$ over $\mathcal{R}_M$ if one wishes to use linear programming relaxations. However, as seen in Theorem \ref{all relaxations}, both  McCormick relaxations are weaker than the SOCP relaxations. In the context of meshed systems, in our preliminary experiments, the difference in quality of bounds produced by the LP and SOCP based relaxations is quite significant. 

As stated in Section \ref{sec:intro}, the goal of this paper is to avoid using SDP relaxations. However, it is interesting to observe the relative strength of different SDP relaxations. On the one hand, $\mathcal{R}_{SDP}$ is the best relaxation among the relaxation considered in Theorem \ref{all relaxations}. On the other hand, quite remarkably, $\mathcal{A}_{SDP}$ is weaker than $\mathcal{A}^{*}_{SOCP}$. Clearly, if one chooses to use SDPs, $\mathcal{A}_{SDP}$ is not a good choice. One may define a different SOCP relaxation applied to the SDP relaxation of the alternative formulation by relaxing constraint (\ref{psd cone}) and replacing it with $2\times2$ principle submatrices. Such a relaxation would yield very poor bounds and thus undesirable.

We note that \cite{coffrin2014} and \cite{coffrin2015} show that a relaxation with same bound as $\mathcal{R}_{SOCP}$ is quite strong. The equality $\mathcal{R}_{SOCP} = \mathcal{A}^{*}_{SOCP}$ is a straightforward observation. However, between these two relaxations, working with the classic SOCP relaxation $\mathcal{A}^{*}_{SOCP}$ provides a very natural way to strengthen these SOCP relaxations. In particular, $\mathcal{A}^{*}_{SOCP}$ was obtained by first dropping the nonconvex arctangent constraint (\ref{arctan cons}). If one is able to incorporate LP/SOCPs based convex outer approximations of these constraints, then $\mathcal{A}^{*}_{SOCP}$ could be significantly strengthened. Indeed one may be even able to produce relaxations that are incomparable to $\mathcal{R}_{SDP}$. We show how to accomplish this in the next section.

\section{Strong SOCP Relaxation for Meshed Networks}
\label{section:meshed}


In this section, we propose three methods to strengthen the classic SOCP relaxation $\mathcal{A}^*_{SOCP}$. In Section \ref{deal cycles}, we propose a new  relaxation of the arctangent constraint of the arctangent constraints \eqref{arctan cons} as polynomial constraints over cycles in the power network. These polynomial constraints with degrees proportional to the length of the cycles are then transformed to systems of bilinear equations by triangulating the cycles. We apply  McCormick relaxation to the bilinear constraints. The resulting convex relaxation is incomparable to the standard SDP relaxation, i.e. the former does not dominate nor is dominated by the latter. In Section \ref{sec:arctan envelope}, we construct a polyhedral envelope for the arctangent functions in 3-dimension and incorporate it into the classic SOCP relaxation, which again results in a convex relaxation incomparable to the SDP relaxation. In Section \ref{SOCP SDP sep}, we strengthen the classic SOCP relaxation by dynamically generating valid linear inequalities that separate the SOCP solution from the SDP cones. This approach takes the advantage of the efficiency of the SOCP relaxation and the accuracy of the SDP relaxation. It very rapidly produces solutions of quality extremely close to the SDP relaxation. In Section \ref{sec:bounding}, we propose variable bounding techniques that provide tight variable bounds for the first two strengthening approaches.

\subsection{A New Cycle-based Relaxation of OPF}
\label{deal cycles}
Our first method is based on the following observation: instead of satisfying the angle condition \eqref{arctan cons} for each $(i,j) \in \mathcal{L}$, we consider a relaxation that guarantees angle differences sum up to 0 modulo $2\pi$ over every cycle $C$ in the power network, i.e.
\begin{equation}
\sum_{(i,j)\in C } \theta_{ij} = 2\pi k, \quad \text{ for some } k \in \mathbb{Z}, \label{angle = 0} 
\end{equation}
where $\theta_{ij}$ is the angle difference between adjacent buses $i$ and $j$.
Although the number of cycles in a power network may be large, it suffices to enforce \eqref{angle = 0} only over cycles in a cycle basis. For a formal definition of cycle basis, see e.g., \cite{Kavitha2009}. Since $\theta = 2\pi k \text{ for some } k \in \mathbb{Z} \iff \cos \theta = 1$, we can equivalently write  (\ref{angle = 0}) as follows:
\begin{equation}
\cos \biggl ( \sum_{(i,j)\in C } \theta_{ij}  \biggr ) = 1. \label{cos = 1} 
\end{equation}
We call \eqref{cos = 1} \emph{cycle constraints}. 
By expanding the cosine term appropriately, we can express \eqref{cos = 1} in terms of $\cos(\theta_{ij})'s$ and $\sin(\theta_{ij})$'s. According to the construction of the alternative formulation \eqref{SOCP}, we have the following relationship between $c,s$ and $\theta$
\begin{equation} \label{cos sine def}
\cos(\theta_{ij}) = \frac{c_{ij}}{\sqrt{c_{ii}c_{jj}}} \  \text{ and } \  \sin(\theta_{ij}) = -\frac{s_{ij}}{\sqrt{c_{ii}c_{jj}}} \quad \forall (i,j)\in\mathcal{L}. 
\end{equation}
Using \eqref{cos sine def}, the cycle constraint \eqref{cos = 1} can be reformulated as a degree $|C|$ homogeneous polynomial equality in the $c_{ii}$, $c_{ij}$, and $s_{ij}$ variables. Denote it as $p_{|C|}$. 

Unfortunately, directly solving the polynomial relaxation can be intractable, especially for large cycles, since $p_{|C|}$ can have up to $2^{|C|-1}+1$ monomials and each monomial of degree $|C|$. It is well known that any polynomial constraint can be written as a set of quadratic constraints using additional variables. In our case, there is a natural way to obtain a  $O(|C|)$ sized system of bilinear constraints by decomposing large cycles into smaller ones. Before introducing the cycle decomposition method, we present two building blocks in the construction, namely, the cycle constraints over 3- and 4-cycles.

\subsubsection{3-Cycle.} \label{sec:3 cycle} 
Let us first analyze the  simplest case, namely the cycle constraint over a 3-cycle. Expansion of $\cos (\theta_{12}+ \theta_{23} + \theta_{31} ) = 1$ gives $p_3=0$, where $p_3$ is the following cubic polynomial:
\begin{equation}
p_3=c_{12}(c_{23}c_{31}-s_{23}s_{31})-s_{12}(s_{23}c_{31}+c_{23}s_{31})-c_{11}c_{22} c_{33}. \label{cos sum 3}
\end{equation}
Note that if the entire power network is a 3-cycle, then adding \eqref{cos sum 3} to the alternative formulation \eqref{SOCP} makes it exactly equivalent to the rectangular formulation \eqref{rect}.
Now we claim that $p_3$ can be replaced with two bilinear constraints. To start with, let us define the polynomials
\begin{equation} 
p_{ij} = c_{ij}^2+s_{ij}^2-c_{ii}c_{jj}, \quad (i,j) \in \{(1,2) , (2,3), (3,1)\}
\end{equation}
and 
\begin{subequations} \label{bilinear 3cycle}
\begin{align}
q_3^1&=s_{12}c_{33}+ c_{23}s_{31}+s_{23}c_{31}  \\
q_3^2&=c_{12}c_{33}-  c_{23}c_{31}+s_{23}s_{31} .
\end{align}
\end{subequations}
Then, we have the following result.
\begin{prop} \label{prop:p3}
$ \{(c,s): p_3 = p_{12}=p_{23}=p_{31} = 0\} = \{(c,s):q_3^1=q_3^2=p_{12}=p_{23}=p_{31} = 0\}$.
\end{prop}
\begin{proof} We prove the above two sets are equal by showing they contain each other.
\begin{enumerate}
	\item ($\subseteq$): $p_{12}=p_{23}=p_{31} = 0$ ensures that there exist $\theta_{12}$, $\theta_{23}$, and $\theta_{31}$ such that the $c,s,\theta$ variables satisfy \eqref{cos sine def}. Then, using the fact that \eqref{cos sum 3} is derived using \eqref{cos sine def} and \eqref{cos = 1}, $p_3 = 0$ implies that 
	\begin{equation} \label{angle 3 def}
	\theta_{12} = - (\theta_{23} +  \theta_{31} ) + 2k\pi \quad\forall k\in\mathbb{Z},
	\end{equation}
which is equivalent to the following two equalities,
	\begin{align} \label{eq:3cycle-sincos}
		\sin(\theta_{12}) = \sin(-(\theta_{23}+\theta_{31}))\; \text{  and  } \; \cos(\theta_{12}) = \cos(-(\theta_{23}+\theta_{31})).
	\end{align}
	Note that the sine constraint in \eqref{eq:3cycle-sincos} implies that 
	\begin{align*}
		\theta_{12} = -(\theta_{23}+\theta_{31}) + 2k\pi \;\text{  or  }\; \theta_{12} = (\theta_{23}+\theta_{31}) + (2k+1)\pi \quad\forall k\in\mathbb{Z},
	\end{align*} 
	while the cosine constraint in \eqref{eq:3cycle-sincos} implies that
	\begin{align*}
		\theta_{12} = -(\theta_{23}+\theta_{31}) + 2k\pi \;\text{  or  }\; \theta_{12} = (\theta_{23}+\theta_{31}) + 2k\pi \quad\forall k\in\mathbb{Z}.
	\end{align*} 
	Therefore, we need both the sine and cosine constraints in \eqref{eq:3cycle-sincos} to enforce \eqref{angle 3 def}. Then, 
	expanding the sine and cosine constraints in \eqref{eq:3cycle-sincos} using the sum formulas and replacing the trigonometric terms with their algebraic equivalents in terms of the $c$ and $s$ variables  via \eqref{cos sine def}, we obtain $q_3^1=q_3^2=0$.
	
	\item ($\supseteq$):  
	$q_3^1=q_3^2=p_{12}=p_{23}=p_{31} = 0$ imply that we have
	\begin{align*}
	p_3 &=c_{12}(c_{23}c_{31}-s_{23}s_{31})-s_{12}(s_{23}c_{31}+c_{23}s_{31})  -c_{11}c_{22} c_{33}  & \\
	&= c_{12} (c_{12}c_{33}) - s_{12}(-s_{12}c_{33}) -c_{11}c_{22} c_{33}    &\text{(due to  } & q_3^1=0  \text{ and } q_3^2=0 ) \\
	&= (c_{12} c_{12}+ s_{12}s_{12})c_{33} -c_{11}c_{22} c_{33} &  \\
	&= 0.  &\text{  (due to }& p_{12}=0 )
	\end{align*}
\end{enumerate}
This completes the proof. 
\end{proof}

Finally, we note that one can obtain two more pairs of equalities constructed in the same fashion as \eqref{bilinear 3cycle} by considering other permutations such as
$
\theta_{23} = - ( \theta_{12}+  \theta_{31} ) + 2k\pi
${ and } 
$
\theta_{31} = - ( \theta_{12}+ \theta_{23} ) + 2k\pi.
$

\subsubsection{4-Cycle.}\label{sec:4 cycle} 

Now, let us analyze the cycle constraint over a  4-cycle. Expansion of $\cos (\theta_{12} + \theta_{34}+ \theta_{23} + \theta_{41}) = 1$ together with \eqref{cos sine def} gives $p_4=0$, where $p_4$ is the following quartic polynomial,
\begin{equation}
p_4=(c_{12}c_{34}-s_{12}s_{34})(c_{23}c_{41}-s_{23}s_{41})  -(s_{12}c_{34}+c_{12}s_{34})(s_{23}c_{41}+c_{23}s_{41})  -c_{11}c_{22} c_{33} c_{44}. \label{cos sum 4}
\end{equation}
We again claim that $p_4$ can be replaced with two bilinear constraints. To start with, define the following polynomials
\begin{equation*} 
p_{ij} = c_{ij}^2+s_{ij}^2-c_{ii}c_{jj}, \quad (i,j) \in \{(1,2) , (2,3), (3,4), (4,1)\}
\end{equation*}
and 
\begin{subequations}\label{bilinear 4-cycle}
\begin{align}
q_4^1 &= s_{12}c_{34}+ c_{12}s_{34}+s_{23}c_{41} +c_{23}s_{41}   \\
q_4^2 &=c_{12}c_{34}- s_{12}s_{34} - c_{23}c_{41}+s_{23}s_{41} .
\end{align}
\end{subequations}
Then, we have the following proposition.
\begin{prop}\label{prop:p4}
$ \{(c,s): p_4 = p_{12}=p_{23}=p_{34}=p_{41} = 0\} = \{(c,s):q_4^1=q_4^2=p_{12}=p_{23}=p_{34}=p_{41} = 0\}$.
\end{prop}
\begin{proof} The proof is similar to that of Proposition \ref{prop:p3}.
\begin{enumerate}
	\item ($\subseteq$):  $p_4 = p_{12}=p_{23}=p_{34}=p_{41} = 0 $ imply that we can find $\theta_{12}$, $\theta_{23}$, $\theta_{34}$ and $\theta_{41}$ satisfying
	\begin{equation} \label{angle 4 def}
	\theta_{12}  + \theta_{34} = - ( \theta_{23} +   \theta_{41}) + 2k\pi \quad\forall k\in\mathbb{Z}.
	\end{equation}
	Then, by taking sine and cosine of both sides,  expanding the right hand side using sine and cosine sum formulas and replacing the trigonometric terms with their algebraic equivalents in terms of $c$ and $s$ variables via \eqref{cos sine def}, we obtain $q_4^1=q_4^2=0$.
		
	\item ($\supseteq$):
	$q_4^1=q_4^2=p_{12}=p_{23}=p_{34}=p_{41} = 0$ imply that we have
	\begin{align*}
	p_4 &=
	(c_{12}c_{34}-s_{12}s_{34})(c_{23}c_{41}-s_{23}s_{41})  -(s_{12}c_{34}+c_{12}s_{34})(s_{23}c_{41}+c_{23}s_{41})  -c_{11}c_{22} c_{33} c_{44} \\
	&=(c_{23}c_{41}-s_{23}s_{41})^2+ (s_{23}c_{41}+c_{23}s_{41})^2  -c_{11}c_{22} c_{33} c_{44} \hspace{1.0cm} (\text{due to } q_4^1=0  \text{ and } q_4^2=0) \\ 
	&= (c_{23}^2+s_{23}^2)(c_{41}^2+s_{41}^2) -c_{11}c_{22} c_{33} c_{44}  \\
	&= 0.   \hspace{9.1cm} (\text{due to } p_{23}=0  \text{ and } p_{41}=0)
	\end{align*}
\end{enumerate}
This completes the proof.

 \end{proof}

\subsubsection{Larger Cycles.}\label{sec:largercycles} Now we introduce the cycle decomposition procedure so that the cycle constraint \eqref{angle = 0} over any cycle $C$ can be reformulated as a system of bilinear equalities, where the number of bilinear equations is $O(|C|)$. Let $C$ be a cycle with buses numbered from 1 to $n$.



\noindent {\bf 3-decomposition of a cycle $C$.} Suppose $|C|=n \ge 4$. We can decompose $C$ into 3-cycles by creating artificial edges $(1,i)$ for $ i=3, \dots, n-1$. Now, we apply the exact reformulation from Section \ref{sec:3 cycle} for each of these cycles. The polynomial cycle constraint $ p_{n}$ is replaced by the following set of bilinear equalities:
\begin{subequations} \label{3-decomposition}
\begin{align}
&\tilde s_{1,i}c_{i+1, i+1}+s_{i,i+1}\tilde c_{1,i+1}-\tilde s_{1,i+1}c_{i,i+1}=0 & i=2,3,\dots,n-1 \label{3-decomposition sine} \\
&\tilde c_{1,i}c_{i+1,i+1}-c_{i,i+1}\tilde c_{1,i+1}-\tilde s_{1,i+1}s_{i,i+1}=0 & i=2,3,\dots,n-1 \label{3-decomposition cosine} \\
&\tilde c_{1,i}^2 + \tilde s^2_{1,i} = c_{11}c_{ii} & i=2,3,\dots,n-1 \label{3-decomposition cone} \\
& c_{ij}^2 + s^2_{ij} = c_{ii}c_{jj} & (i,j)\in C.
\end{align}
\end{subequations}
Here, $\tilde c_{1i}$ and $\tilde s_{1i}$ are extra variables representing $\sqrt{c_{11}c_{ii}}\cos(\theta_1-\theta_i)$ and  $-\sqrt{c_{11}c_{ii}}\sin(\theta_1-\theta_i)$, respectively, for  $ i=3, \dots, n-1$,
and $\tilde c_{1i}$ and $\tilde s_{1i}$ for $i=2$ and $i=n$ coincidence with the original variables $c_{1i}, s_{1i}$.

\begin{prop} \label{prop:3-decomp}
Suppose a cycle basis is given for the power network $\mathcal{N} = (\mathcal{B}, \mathcal{L})$. 
Then, constraints (\ref{bilinear 3cycle}) for every 3-cycle in the cycle basis and constraints (\ref{3-decomposition}) for every cycle with length $n\geq 4$ in the cycle basis 
define a valid bilinear extended relaxation of OPF  \eqref{SOCP}-\eqref{arctan cons}. Moreover, it implies that $\sum_{(i,j) \in C} \atan2(s_{ij},c_{ij}) = 2\pi k$ for some integer $k$ for each cycle in the cycle basis. 
\end{prop}
\begin{proof}
First, let us prove that the proposed relaxation is valid for any feasible solution of the OPF formulation over any cycle $C$ with length $|C|=n\ge4$ in the cycle basis. Without loss of generality, assume that the buses in the cycle are numbered from 1 to $n$. Let $(c, s, \theta)$ be a feasible solution for OPF \eqref{SOCP}-\eqref{arctan cons}. Recall that that we have $c_{i,i+1}=\sqrt{c_{ii}c_{i+1,i+1}}\cos(\theta_i-\theta_{i+1})$ and $s_{i,i+1}=-\sqrt{c_{ii}c_{i+1,i+1}}\sin(\theta_i-\theta_{i+1})$ for $i=1,\dots,n-1$. Then, choose $\tilde c_{1i} =\sqrt{c_{11}c_{ii}}\cos(\theta_1-\theta_i)$, $\tilde s_{1i}=-\sqrt{c_{11}c_{ii}}\sin(\theta_1-\theta_i)$ for each artificial line $(1,i)$ for $i=2, \dots, n-1$. Now, we have
\begin{align*}
& \quad \tilde s_{1,i}c_{i+1, i+1}+s_{i,i+1}\tilde c_{1,i+1}-\tilde s_{1,i+1}c_{i,i+1} 
\\ &= [-\sqrt{c_{11}c_{ii}}\sin(\theta_1-\theta_i)] c_{i+1, i+1} + [-\sqrt{c_{ii}c_{i+1,i+1}}\sin(\theta_i-\theta_{i+1}] [\sqrt{c_{11}c_{ii}}\cos(\theta_1-\theta_i)] \\ & \quad+   [- \sqrt{c_{11}c_{i+1,i+1}}\sin(\theta_1-\theta_{i+1})] [\sqrt{c_{ii}c_{i+1,i+1}}\cos(\theta_i-\theta_{i+1})] 
\\& =  c_{i+1,i+1}\sqrt{c_{11}c_{ii}} \bigl[\sin(\theta_1-\theta_i) - \sin(\theta_i-\theta_{i+1})\cos(\theta_1-\theta_i) + \cos(\theta_i-\theta_{i+1})\sin(\theta_1-\theta_{i+1} \bigr] 
 \\&= 0,
\end{align*}
which proves the validity of \eqref{3-decomposition sine}. A similar argument can be used to prove the validity of \eqref{3-decomposition cosine} as well.
Also, since we have $\tilde c_{1,i}^2 + \tilde s^2_{1,i} - c_{11}c_{ii} = (\sqrt{c_{11}c_{ii}}\cos(\theta_1-\theta_i))^2+(-\sqrt{c_{11}c_{ii}}\sin(\theta_1-\theta_i))^2- c_{11}c_{ii}= c_{11}c_{ii} [\cos^2(\theta_1-\theta_i)+\sin^2(\theta_1-\theta_i) - 1]=0$,  \eqref{3-decomposition cone} follows. Hence, constraints (\ref{3-decomposition}) for every cycle with length $n\geq 4$ are valid for OPF \eqref{SOCP}-\eqref{arctan cons}.

For the second part, it is sufficient to show that adding (\ref{3-decomposition}) to \eqref{SOCP} implies $\cos \left ( \sum_{i=1}^{n-1} \theta_{i,i+1}  +\theta_{n1}   \right ) = 1$ for every cycle $C$ with length $n\geq 4$.  
Using the argument in Proposition \ref{prop:p3}, we know that the equalities in \eqref{3-decomposition} and \eqref{coupling} enforce the cycle constraint over the 3-cycle $(1,i,i+1)$
\begin{align}\label{3-decomp angles}
\theta_{1,i} + \theta_{i,i+1}+\theta_{i+1,1} = 2\pi k_i \quad\quad i=2,\cdots, n-1,
\end{align}
for some integers $k_i$. 
Therefore, summing \eqref{3-decomp angles} over all $i$ and canceling $\theta_{1,i}+\theta_{i,1}=0$, we conclude that $\sum_{i=1}^{n-1} \theta_{i,i+1}  +\theta_{n1}  = 2\pi k$, for some $k \in \mathbb{Z} $.


 \end{proof}

\noindent {\bf 4-decomposition of a cycle $C$.} 
Suppose $|C| \ge 5$ and odd. We can  decompose the cycle $C$ into 4-cycles by creating artificial edges $(1, 2i)$ for $  i=2, 3, \dots, \frac{ n-1}{2}$ and one 3-cycle. Now, we apply the exact reformulation from Section \ref{sec:3 cycle} and \ref{sec:4 cycle}  for each of these cycles. Finally, polynomial $ p_{n}$ is replaced by the following set of bilinear equalities:
\begin{subequations} \label{4-decomposition odd}
\begin{align}
&\tilde c_{1,2i-2}c_{2i-1,2i}-\tilde s_{1,2i-2}s_{2i-1,2i} - \tilde c_{1,2i}c_{2i-2,2i-1} - \tilde s_{1,2i}s_{2i-2,2i-1} = 0    & i=2,\dots, \frac{n-1}{2} \label{4-decomp odd 41}\\ 
&\tilde s_{1,2i-2}c_{2i-1,2i}+\tilde c_{1,2i-2}s_{2i-1,2i} - \tilde s_{1,2i}c_{2i-2,2i-1} + \tilde c_{1,2i}s_{2i-2,2i-1} = 0     & i=2,\dots, \frac{n-1}{2} \label{4-decomp odd 42} \\ 
&\tilde c_{1,n-1}c_{n-1,n} - \tilde s_{1,n-1}s_{n-1,n}  - c_{n,1} c_{n-1,n-1} = 0  & \label{4-decomp odd 31} \\ 
&\tilde s_{1,n-1}c_{n-1,n} +\tilde c_{1,n-1}s_{n-1,n}   +s _{n,1} c_{n-1,n-1} = 0 & \label{4-decomp odd 32} \\
&\tilde s_{1,i}^2 + \tilde c^2_{1,i} = c_{11}c_{ii} &\qquad i=2,3,\dots,n-1\\
& c_{ij}^2 + s^2_{ij} = c_{ii}c_{jj} & (i,j)\in C,
\end{align}
\end{subequations}
where \eqref{4-decomp odd 41}-\eqref{4-decomp odd 42} are constraints on 4-cycles and \eqref{4-decomp odd 31}-\eqref{4-decomp odd 32} are constraints on the last 3-cycle, $\tilde c_{1,2i}, \tilde s_{1,2i}$ for $i=2, 3, \dots, (n-1)/2$ are additional variables and $\tilde c_{12}, \tilde s_{12}$ coincide with the original variables $c_{12}, s_{12}$. 

%

Suppose $|C| \ge 6$ and even.  We can  decompose the cycle $C$ into 4-cycles by creating the artificial edges $(1, 2i)$ for $ i=2,\dots, \frac{n-2}{2}$. Now, we apply the exact reformulation from Section  \ref{sec:4 cycle}  for each of these cycles. Finally, polynomial $ p_{n}$ is replaced by the following set of bilinear equalities:
\begin{subequations} \label{4-decomposition even}
\begin{align}
& \tilde c_{1,2i-2}c_{2i-1,2i}-\tilde s_{1,2i-2}s_{2i-1,2i} - \tilde c_{1,2i}c_{2i-2,2i-1} - \tilde s_{1,2i}s_{2i-2,2i-1} = 0   & i=2,\dots, \frac{n}{2} \\ 
&\tilde s_{1,2i-2}c_{2i-1,2i}+\tilde c_{1,2i-2}s_{2i-1,2i} - \tilde s_{1,2i}c_{2i-2,2i-1} + \tilde c_{1,2i}s_{2i-2,2i-1} = 0    & i=2,\dots, \frac{n}{2}\\
&\tilde s_{1,i}^2 + \tilde c^2_{1,i} = c_{1,1}c_{i,i} & i=2,3,\dots,n-1  \\
& c_{ij}^2 + s^2_{ij} = c_{ii}c_{jj} & (i,j)\in C,
\end{align}
\end{subequations}
where $\tilde c_{1,2i}, \tilde s_{1,2i}$ are additional variables for $i=2,\dots,n/2-1$, and $\tilde c_{1,n}, \tilde s_{1,n}$ coincide with the original variables $c_{1,n}, s_{1,n}$. 

\begin{prop}\label{prop:4-decomp}
Suppose a cycle basis is given for the power network $\mathcal{N} = (\mathcal{B}, \mathcal{L})$. Then,
\begin{enumerate}
\item Constraints \eqref{bilinear 3cycle} for every 3-cycle in the cycle basis,
\item Constraints \eqref{bilinear 4-cycle} for every 4-cycle in the cycle basis,
\item Constraints \eqref{4-decomposition odd} for every odd cycle with length $n\geq 5$ in the cycle basis,
\item Constraints \eqref{4-decomposition even} for every even cycle with length $n\geq 6$ in the cycle basis,
\end{enumerate}
define a valid bilinear extended relaxation of OPF \eqref{SOCP}-\eqref{arctan cons}. Moreover, it implies that $\sum_{(i,j) \in C} \atan2(s_{ij},c_{ij}) = 2\pi k$ for some integer $k$ for each cycle in the cycle basis. 
\end{prop}
The proof is similar to the proof of Proposition \ref{prop:3-decomp}.
%
%

\subsubsection{McCormick Based LP Relaxation and Separation.}\label{sec:McCormickSeparation}
The cycle-based OPF formulations presented in Proposition \ref{prop:3-decomp} and Proposition \ref{prop:4-decomp} are non-convex quadratic problems, for which we can obtain LP relaxation by using McCormick relaxations of the bilinear constraints over cycles. For large networks, including McCormick relaxations for all the cycle constraints may be computationally inefficient. Therefore, we propose a separation routine which generates cutting planes to separate a solution of the classic SOCP relaxation from the McCormick envelopes of the cycle constraints. The separation is applied to every cycle in the cycle basis individually.

For a given cycle $C$, the McCormick relaxation of the bilinear cycle constraints, which could be any one of:  (\ref{bilinear 3cycle}), (\ref{bilinear 4-cycle}), \eqref{3-decomposition}, \eqref{4-decomposition odd}, or \eqref{4-decomposition even}, can be written compactly as follows:
\begin{subequations}\label{lp system}
\begin{align}
& Az + \tilde A \tilde z + B y \le c \label{mccormick and bounds} \\
& Ey = 0, \label{cycle equalities} 
\end{align}
\end{subequations}
where $z$ is a vector composed of the $c,s$ variables in the alternative formulation \eqref{SOCP}, $\tilde z$ is a vector composed of the additional $\tilde c, \tilde s$ variables introduced in the cycle decomposition, and $y$ is a vector of new variables defined to linearize the bilinear terms in the cycle constraints. Constraint \eqref{mccormick and bounds} contains the McCormick envelopes of the bilinear terms and bounds  on the $c,s$ variables, while \eqref{cycle equalities} includes the linearized cycle equality constraints.

Given an optimal solution $z^*$ of the classic SOCP relaxation $\mathcal{A}^*_{SOCP}$, we can solve the following separation problem for a cycle $C$,
\begin{subequations} \label{separation LP}
\begin{align}
	v^*:=\min_{\alpha, \beta, \gamma, \mu, \lambda} \quad & \beta - \alpha^T z^* \\
	\text{s.t.}\quad 
					 & A^T \lambda = \alpha \label{eq:sep LP1}\\
				     & \tilde A^T \lambda = 0 \\
	                 & B^T \lambda + E^T \mu = 0 \\
	                 & c^T \lambda \le \beta, \;\;\lambda \geq 0 \label{eq:sep LP2}\\
					 & -e\leq \alpha \leq e \label{eq:sep LPbound1}\\
					 & -1\le \beta\le 1, \label{eq:sep LPbound2}
\end{align}
\end{subequations}
where \eqref{eq:sep LP1}-\eqref{eq:sep LP2} is the dual system equivalent to the condition that $\alpha^T z\leq \beta$ for all $(z,\tilde z, y)$ satisfies \eqref{lp system}; \eqref{eq:sep LPbound1}-\eqref{eq:sep LPbound2} bounds the coefficients $\alpha, \beta$, and $e$ is the vector of $1$'s. If $v^*<0$, then the corresponding optimal solution $(\alpha,\beta)$ of \eqref{separation LP} gives a separating hyperplane such that $\alpha^Tz^*>\beta$ and $\alpha^T z\leq \beta$ for all $(z,\tilde z, y)$ in \eqref{lp system}. If $v^*\ge 0$, then \eqref{separation LP} certifies that $(z^*, \tilde z, y)$ is contained in the McCormick relaxation \eqref{lp system} for some $\tilde z, y$.


We remark that the McCormick relaxations obtained from 3-decomposition and 4-decomposition of cycles do not dominate one another. Therefore, we use both of them in the separation routine. We also note that the classic SOCP relaxation strengthened by dynamically adding valid inequalities through separation over the McCormick relaxations of the cycle constraints is incomparable to the standard SDP relaxations of OPF. 
\begin{prop} \label{cycle incomp}
	$\mathcal{A}^*_{SOCP}$ strengthened by the valid inequalities from the McCormick relaxation of the cycle constraints is not dominated by nor dominates $\mathcal{R}_{SDP}$.
\end{prop}
This result is verified by an example in Section \ref{incomparable}.

\subsection{Arctangent Envelopes} \label{sec:arctan envelope}
In this section, we propose the second approach to strengthen the classic SOCP relaxation $\mathcal{A}^*_{SOCP}$. The key idea is still to incorporate convex approximation of the angle condition \eqref{arctan cons} to the SOCP relaxation. This time, instead of reformulating polynomial constraints over cycles such as \eqref{3-decomposition}, \eqref{4-decomposition odd}, and \eqref{4-decomposition even}, we propose linear envelopes for the arctangent function over a box, and incorporate this relaxation to $\mathcal{A}^*_{SOCP}$.

Our construction uses four linear inequalities to approximate the convex envelope for the following set defined by the arctangent constraint \eqref{arctan cons} for each line $(i,j)\in\mathcal{L}$, 
\begin{equation}
\mathcal{AT}:=\left\{(c,s,\theta)\in\mathbb{R}^3 : \theta = \arctan \left ( \frac sc \right), (c,s)\in[\underline c, \overline c]\times[\underline s, \overline s] \right\},	\label{arctan constraint}
\end{equation}
where we denote $\theta= \theta_j - \theta_i$ and drop $(i,j)$ indices for brevity. We also assume $\underline c>0$. The four corners of the box correspond to four points in the $(c,s,\theta)$ space: 
\begin{subequations}
\begin{align}
   z^1 &= (\underline c, \overline s, \arctan \left( {\overline s}/{\underline c} \right) )\\
   z^2 &= (\overline c, \overline s, \arctan \left( {\overline s}/{\overline c} \right) )\\ 
   z^3 &= (\overline c,  \underline s, \arctan \left( {\underline s}/{\overline c} \right))\\
   z^4 &= (\underline c, \underline s, \arctan \left( {\underline s}/{\underline c} \right)). 
\end{align}
\end{subequations}

Two inequalities that approximate the upper envelop of $\mathcal{AT}$ are described below.
\begin{prop}
Let $\theta = \gamma_1 + \alpha_1 c + \beta_1 s$  and $\theta = \gamma_2 + \alpha_2 c + \beta_2 s$ be the  planes passing through points $\{z^1,z^2, z^3\}$,  and $\{z^1,z^3, z^4\}$, respectively.
Then, two valid inequalities for $\mathcal{AT}$ can be obtained as
\begin{align}\label{eq:arctan-upper}
\gamma_k' + \alpha_k c + \beta_k s \ge \arctan \left( \frac{s}{c} \right)
\end{align}
for all $(c,s) \in[\underline c, \overline c]\times[\underline s, \overline s]$ with $\gamma_k' = \gamma_k + \Delta \gamma_k$, where
\begin{equation} \label{error problem}
\Delta \gamma_k =   \max \left \{  \arctan \left( \frac{s}{c} \right) - (\gamma_k+ \alpha_k c + \beta_k s ) :   c \in [\underline c, \overline c], \ s \in [\underline s, \overline s] \right \}, 
\end{equation} 
for $k=1,2$.
\end{prop}
Note that by the construction of \eqref{error problem}, it is evident that $\gamma_k'+\alpha_k c + \beta_k s$ dominates the $\arctan(s/c)$ over the box. The nonconvex optimization problem  \eqref{error problem} can be solved by enumerating all possible Karush-Kuhn-Tucker (KKT) points.
%
%
%
%
%
%

\begin{figure}[t!]
	\centering    
	\caption{Arctangent envelopes from different viewpoints. Red planes are the envelopes.}  \label{arctan figures}
	\subfigure[Lower envelopes.]{\label{fig:a}\includegraphics[width=60mm]{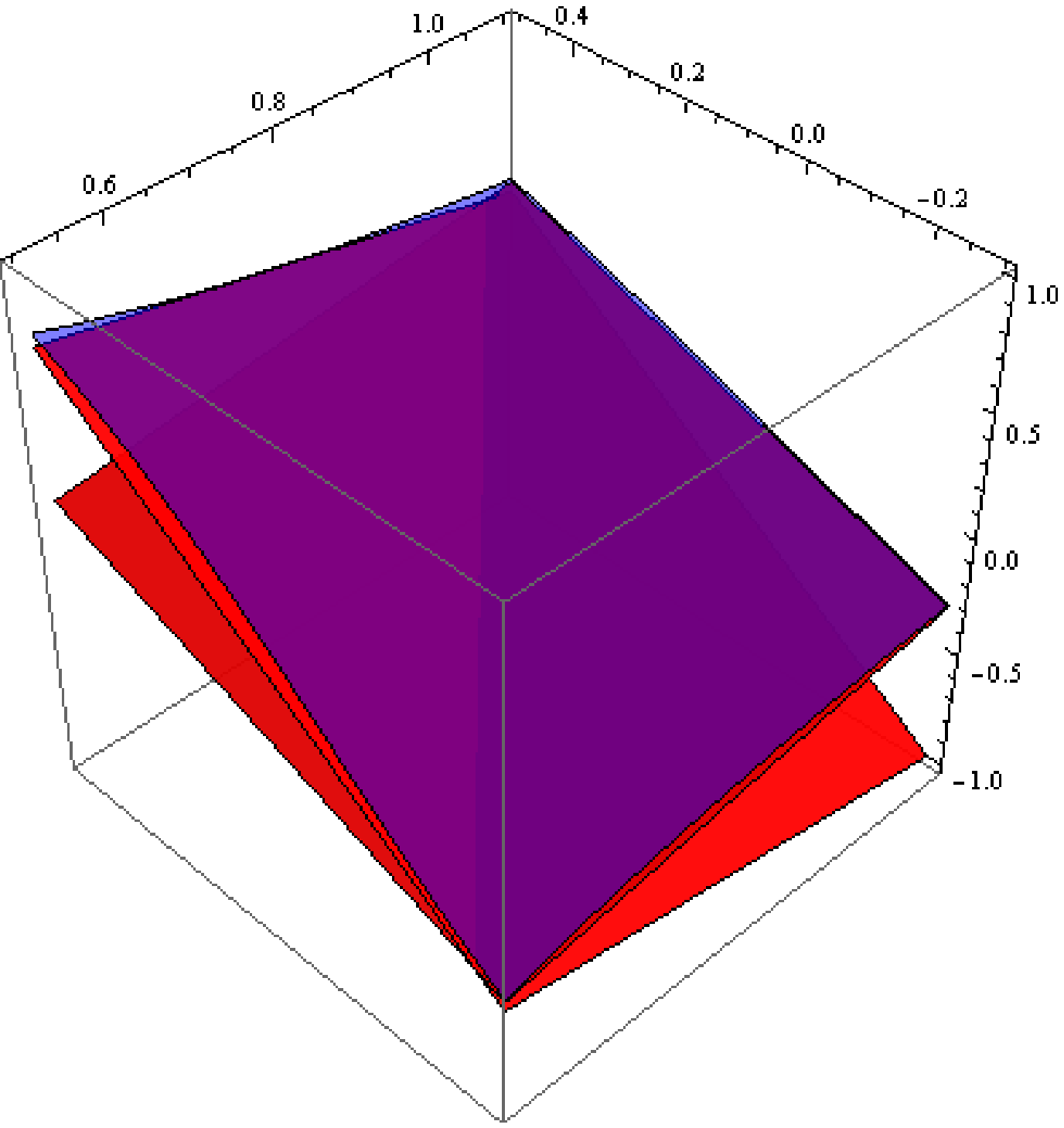}}
	\subfigure[Upper envelopes.]{\label{fig:b}\includegraphics[width=60mm]{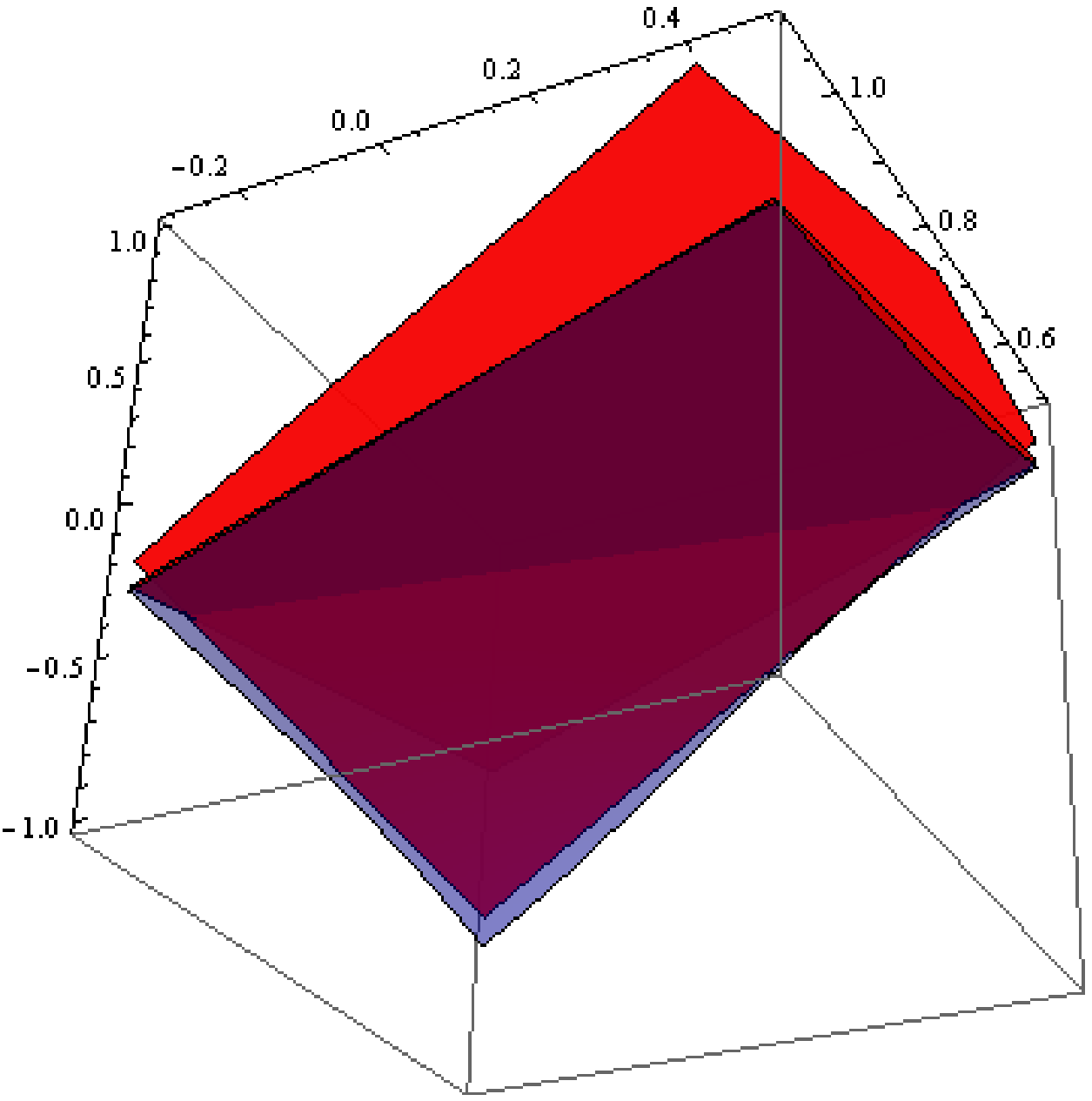}}
\end{figure}

Two inequalities that approximate the lower envelop of $\mathcal{AT}$ are described below.

\begin{prop}
Let $\theta = \gamma_3 + \alpha_3 c + \beta_3 s$  and $\theta = \gamma_4 + \alpha_4 c + \beta_4 s$ be the  planes passing through points $\{z^1,z^2, z^4\}$,  and $\{z^2,z^3, z^4\}$, respectively.
Then, two valid inequalities for  $\mathcal{AT}$ are defined as
\begin{align}\label{eq:arctan-lower}
	\gamma_k' + \alpha_k c + \beta_k s \le \arctan \left( \frac{s}{c} \right)
\end{align}
for all $(c,s) \in[\underline c, \overline c]\times[\underline s, \overline s]$ with $\gamma_k' = \gamma_k - \Delta \gamma_k$, where 
\begin{equation} \label{eq:arctan-lower-deltagamma}
\Delta \gamma_k =   \max \left \{ (\gamma_k+ \alpha_k c + \beta_k s ) - \arctan \left( \frac{s}{c} \right)   :   c \in [\underline c, \overline c], \ s \in [\underline s, \overline s] \right \}, 
\end{equation} 
for $k=3,4$.
\end{prop}

Figure \ref{arctan figures} shows an example of these upper and lower envelopes. One may further strengthen these envelopes via additional inequalities, although the benefit is minimal according to our experiments.

We have the following proposition, whose proof is provided by an example in Section \ref{incomparable}. 
\begin{prop} \label{arctan incomp}
$\mathcal{A}^*_{SOCP}$ strengthened by arctangent envelopes defined in \eqref{eq:arctan-upper}-\eqref{eq:arctan-lower-deltagamma} is not dominated by nor dominates the SDP relaxation $\mathcal{R}_{SDP}$.
\end{prop}

\subsection{SDP Separation}
\label{SOCP SDP sep}

The last approach we propose to strengthen the classic SOCP relaxation is similar in spirit to the separation approach in Section \ref{sec:McCormickSeparation}, but here, instead of separating over the McCormick relaxations of the cycle constraints \eqref{lp system}, we separate a given SOCP relaxation solution from the feasible region of the SDP relaxation of cycles. In the following, we first explore the relationship between the classic SOCP relaxation $\mathcal{A}^*_{SOCP}$ and the SDP relaxation $\mathcal{R}^r_{SDP}$. Then, we explain the separation procedure over cycles.

Let $x \in \mathbb{R}^{2|\mathcal{B}|}$ be a vector of bus voltages defined as $x=[e; f]$ such that $x_i=e_i$ for $i\in\mathcal{B}$ and $x_{i'}=f_i$ for $i'=i+|\mathcal{B}|$. Observe that if we have a set of $c,s$ variables satisfying the cosine, sine definition \eqref{cos sine def} and a matrix variable $W=xx^T$, then the following linear relationship between $c,s$ and $W$ holds, 
\begin{subequations}\label{eq:sdplin}
\begin{align}
c_{ij} &= e_ie_j+f_if_j= W_{ij} + W_{i'j'} & (i,j)\in {\mathcal{L}}\\ 
s_{ij} &= e_if_j-e_jf_i= W_{ij'} - W_{ji'} & (i,j)\in {\mathcal{L}} \\ 
c_{ii} &= e_i^2+f_i^2= W_{ii} + W_{i'i'},   & i\in\mathcal{B}.
\end{align}
\end{subequations}


Given a solution of the classic SOCP relaxation, denoted as $(p^*, q^*, c^*, s^*)$, if there exists a symmetric matrix $W^*\in\mathbb{R}^{2|\mathcal{B}|\times 2|\mathcal{B}|}$ such that $(c^*, s^*)$ and $W^*$ satisfy the linear system \eqref{eq:sdplin}, then $(p^*, q^*, W^*)$ satisfies the flow conservation and voltage bound constraints \eqref{rect sdp real first}-\eqref{rect sdp real last} (because $(c^*,s^*)$ satisfies \eqref{activeAtBusR}-\eqref{voltageAtBusR}) as well as the generator real and reactive bounds \eqref{activeAtGenerator}, \eqref{reactiveAtGenerator}. If, furthermore, $W^*\succeq 0$, then $(p^*, q^*, W^*)$ is a feasible solution to the standard SDP relaxation $\mathcal{R}^r_{SDP}$,  therefore, optimal for the SDP relaxation. If there does not exists such a $W^*$, then we can add a valid inequality to the SOCP relaxation to separate $(c^*, s^*)$ from the set defined by \eqref{eq:sdplin} and $W\succeq 0$. This procedure can be repeated until the optimal SDP relaxation solution is obtained. 
	
Notice that the above separation procedure requires solving an SDP problem with a matrix of the size equal to the original SDP relaxation, which can be quite time consuming. Instead of separating over the full matrix, or equivalently the entire power network, we can consider a separation only over cycles. {In this way, we effectively use an SDP relaxation to provide a convex approximation of the angle condition \eqref{arctan cons} over a cycle.}

In particular, for any cycle $C$ in the power network, we only consider the equalities in \eqref{eq:sdplin} associated with this cycle and the corresponding submatrix matrix $\tilde W \in\mathbb{R}^{2|C|\times 2|C|}$ of $W$.
Consider the following set $\mathcal{S}$
\begin{align} \label{psd cone system}
\mathcal{S}:=\left\{z \in \mathbb{R}^{2|C|} : \exists \tilde W\in\mathbb{R}^{2|C|\times 2|C|} \text{  s.t. } -z_l + A_l \bullet \tilde W  = 0 \quad\forall l\in L, \quad \tilde W\succeq 0 \right\},
\end{align}
where $z=(c,s)$, the linear equality represents the linear system \eqref{eq:sdplin} restricted to the cycle $C$, and $L$ is the index set for all these equalities; $\bullet$ denotes the Frobenius inner product between matrices. We suppress labeling variables with $C$ for conciseness. It should be understood that the construction is done for each cycle in a cycle basis.


For a given $z^*$, the separation problem over $\mathcal{S}$ can be written as follows,
\begin{subequations} \label{sdp sep}
	\begin{align}
	v^*:=\min  &\hspace{0.5em}  - \alpha^T z^*\\
	\mathrm{s.t.}   &\hspace{0.5em} \sum_{l \in L}  \lambda_l A_l \succeq 0 \label{eq:sdp-sep-dual1} \\
	&\hspace{0.5em} \alpha + \lambda=0 \label{eq:sdp-sep-dual2}\\
	&\hspace{0.5em}  - e \le \alpha \le e. 
	\end{align}
\end{subequations}
Since the system \eqref{psd cone system} is strictly feasible (e.g., $\tilde W=I$ and $z_l=A_l \cdot I$, for $l \in L$), strong duality holds between the primal system in $\mathcal{S}$ and the dual system in \eqref{sdp sep}. In particular, \eqref{eq:sdp-sep-dual1}-\eqref{eq:sdp-sep-dual2} is equivalent to $\max_{z\in\mathcal{S}} \alpha^Tz\leq 0$. Therefore, the solution of \eqref{sdp sep}
either gives a separating hyperplane of the form $\alpha^T z \le 0$ such that $\alpha^T z^* > 0$ and $\alpha^Tz\leq 0 \;\;\forall z\in\mathcal{S}$, or certifies that $z^*\in\mathcal{S}$. 
If the optimal objective value $v^*$ of \eqref{sdp sep} is strictly less than 0, then we add the homogeneous inequality $ \alpha ^T z \le 0$ to the classic SOCP relaxation.
We can now apply this procedure to every element of a cycle basis and resolve SOCP with the added linear inequalities. In computational experiments, we observe that a few iterations of this algorithm give very tight approximations to SDP relaxation of rectangular formulation.

\subsection{Obtaining Variable Bounds} \label{sec:bounding}

The proposed McCormick relaxations of the cycle constraints and the convex envelopes for the arctangent functions are useful only when good variable upper and lower bounds are available for the $c$ and $s$ variables. In this section, we explain how to obtain good bounds, which is the key ingredient in the success of the first two proposed methods.
 
Observe that $c_{ij}$ and $s_{ij}$  do not have explicit variable bounds except the implied bounds due to (\ref{voltageAtBusR})  and (\ref{coupling}) as
\begin{equation}
-\overline V_i^2\overline V_j^2 \le c_{ij}, s_{ij} \le \overline V_i^2\overline V_j^2 \quad (i,j) \in \mathcal{L}.
\end{equation}
However, these bounds may be loose, since it is usually the case that phase angle differences in a power network under normal operation are small, implying  $c_{ij} \approx 1$ and $s_{ij}\approx 0$. Therefore, one can try to improve these bounds. A straightforward approach is to optimize $c_{ij}$ and $s_{ij}$ over the feasible region of the SOCP relaxation as is proposed in \cite{kocuk2014}. However, this procedure can be expensive because we need to solve $4|\mathcal{L}|$ SOCPs, each of the size of the classic SOCP relaxation. 

To be computationally efficient, instead of solving the full size SOCPs to tighten variable bounds, we can obtain potentially weaker bounds by solving a reduced version of the full SOCP relaxation. 
In particular, to find variable bounds for $c_{kl}$ and $s_{kl}$ for some $(k,l) \in \mathcal{L}$, consider the buses which can be reached from either $k$ or $l$ in at most $r$ steps. Denote these buses by a set $ \mathcal{B}_{kl} (r)$. For instance,  $ \mathcal{B}_{kl} (0) = \{k,l\}$, $ \mathcal{B}_{kl} (1) = \delta(k) \cup \delta(l)$, etc. 
{Also define} 
$\mathcal{G}_{kl} (r)=\mathcal{B}_{kl} (r) \cap \mathcal{G}$ and 
$\mathcal{L}_{kl} (r) = 
\{ (i,j) \in \mathcal{L} : i \in \mathcal{B}_{kl} (r) \text{ or }  j \in \mathcal{B}_{kl} (r) \}$. {Consider} the  following SOCP relaxation, 
\begin{subequations}\label{bounding SOCP}
\begin{align}
   &\hspace{0.5em} p_i^g-p_i^d = G_{ii}c_{ii} + \sum_{j \in \delta(i)}[ G_{ij}c_{ij} -B_{ij}s_{ij}]   & i& \in\mathcal{B}_{kl} (r)\\
  & \hspace{0.5em} q_i^g-q_i^d = -B_{ii}c_{ii} + \sum_{j \in \delta(i)}[ -B_{ij}c_{ij} -G_{ij}s_{ij}]  & i& \in \mathcal{B}_{kl} (r) \\
  & \hspace{0.5em} \underline V_i^2 \le c_{ii} \le \overline V_i^2    & i& \in \mathcal{B}_{kl} (r+1) \\
  & \hspace{0.5em}  p_i^{\text{min}}  \le p_i^g \le p_i^{\text{max}}     & i& \in \mathcal{G}_{kl} (r) \\
  & \hspace{0.5em}  q_i^{\text{min}}  \le q_i^g \le q_i^{\text{max}}     & i& \in \mathcal{G}_{kl} (r) \\
  & \hspace{0.5em} c_{ij}=c_{ji}, \ \ s_{ij}=-s_{ji}    &(&i,j) \in \mathcal{L}_{kl} (r)\\
  & \hspace{0.5em}  c_{ij}^2+s_{ij}^2  \le c_{ii}c_{jj}     &(&i,j) \in \mathcal{L}_{kl} (r).
\end{align}
\end{subequations}
Essentially, \eqref{bounding SOCP} is the classic SOCP relaxation applied to the part of the power network within $r$ steps of the buses $k$ and $l$. Note that {\eqref{bounding SOCP} for each edge $(k,l)$} can be solved in parallel, since they are independent of each other. We {observed} that a good tradeoff between accuracy and speed is to select $r=2$. In our experiments, larger values of $r$ improve variable bounds marginally.

For artificial edges, we cannot use the above procedure as they do not appear in the flow balance constraints. Instead, we use 
bounds on the original variables that are computed through \eqref{bounding SOCP}  to obtain some improved bounds for the variables on the artificial edges. Since any large cycle can be decomposed into 3-cycles and/or 4-cycles as shown in Section \ref{sec:largercycles}, we only need to consider 3- and 4-cycles here. Let us start from a 3-cycle. Assume the upper and lower bounds on $c_{12},s_{12},c_{23},s_{23}$ are already known, and we want to tighten the bounds on the artificial edge $c_{13},s_{13}$. Then, the bilinear constraints \eqref{bilinear 3cycle} over the cycle can be written as follows:
\begin{subequations}\label{eq:varbounds}
\begin{align}
 {c_{13}}&=\frac{  c_{12}c_{23}-s_{12}s_{23} }{c_{22}}\label{eq:varbndc} \\
 {s_{13}}&=\frac{ s_{12}c_{23}+c_{12}s_{23}}{ c_{22}}.\label{eq:varbnds}
\end{align}
\end{subequations}
Now, we can obtain variable bounds on $c_{13}, s_{13}$ by bounding the right-hand sides of \eqref{eq:varbndc}-\eqref{eq:varbnds} over the box for $c_{12},s_{12},c_{23},s_{23}$, and $c_{22}$. 
In particular, an upper bound on $c_{13}$ can be computed as 
\begin{equation}
\bar{c}_{13}= \begin{cases} {\hat{c}}_{13} / \underline c_{22} & \text{ if } { \hat{ c}}_{13} > 0 \\  { \hat{c}}_{13} / \overline c_{22} & \text{ if } { \hat{ c}}_{13} \le 0,  \end{cases}
\end{equation}
where
\begin{equation}
{\hat{c}}_{13} = \max\{{c_{12}c_{23} : c_{12}\in[\underline c_{12}, \overline c_{12}], c_{23}\in[\underline c_{23}, \overline c_{23}] }\}
			     - \min\{{s_{12}s_{23} : s_{12}\in[\underline s_{12}, \overline s_{12}], s_{23}\in[\underline s_{23}, \overline s_{23}] }\}.
\end{equation}
 A similar procedure can be applied to obtain lower bounds on $c_{13}$ and $s_{13}$.

For a 4-cycle of buses $1,2,3,4$, assume we have bounds on $c_{12},s_{12},c_{23},s_{23},c_{34},s_{34}$, and want to find variable bounds on the artificial edge $c_{14}, s_{14}$. 
Using the two bilinear constraints for the 4-cycle in \eqref{bilinear 4-cycle} and \eqref{coupling}, we can express $c_{14}$ and $s_{14}$ in terms of the other variables as follows
\begin{subequations}
\begin{align}
{ c_{14}}&=\frac{ c_{12}(c_{23}c_{34}- s_{23}s_{34}) - { s_{12}(s_{23}c_{34}+c_{23}s_{34}) }}{c_{22}c_{33}} \\   
{ s_{14}}&=\frac{ s_{12}(c_{23}c_{34}- s_{23}s_{34}) + { c_{12}}(s_{23}c_{34}+c_{23}s_{34}) }{ c_{22}c_{33}}. 
\end{align}
\end{subequations}
Now, proceed in two steps. (1) Define $a:=c_{23}c_{34}- s_{23}s_{34}$ and $b:=s_{23}c_{34}+c_{23}s_{34}$, and calculate bounds $a,b$ as described for the 3-cycle case. (2) Repeat this process to obtain bounds on $c_{14}, s_{14}$.

\section{Computational Experiments}
\label{comp experiment}

In this section, we present the results of extensive computational experiments on standard IEEE instances available from MATPOWER \citep{Matpower} and instances from NESTA 0.3.0 archive \citep{nesta}. The code is written in the C\# language with Visual Studio 2010 as the compiler. For all experiments, a 64-bit laptop with Intel Core i7 CPU with 2.00GHz processor and 8 GB RAM is used. Time is measured in seconds, unless otherwise stated. Conic interior point solver MOSEK 7.1 \citep{MOSEK} is used to solve SOCPs and SDPs.

\subsection{Methods}
We report the results of the following four algorithmic settings:
\begin{itemize}
\item
$\mathsf{SOCP}$: The classic SOCP formulation $\mathcal{A}^*_{SOCP}$ without any improvement. 
\item
{$\mathsf{SOCPA}$}: $\mathsf{SOCP}$ strengthened {by the} arctangent envelopes {\eqref{eq:arctan-upper}-\eqref{eq:arctan-lower-deltagamma}}.
\item
{$\mathsf{S34A}$}: $\mathsf{SOCPA}$ further strengthened {by dynamically generating linear} valid inequalities from {the McCormick relaxation of the} cycle {constraints via the 3- and 4-cycle decompositions and the separation routine developed in Sections \ref{sec:largercycles} and \ref{sec:McCormickSeparation}}.
\item
{$\mathsf{SSDP}$}: $\mathsf{SOCP}$ strengthened {by dynamically generating linear} valid inequalities obtained from {separating an SOCP feasible solution from the SDP relaxation over cycles. The separation routine is developed in Section \ref{SOCP SDP sep}.} 
\end{itemize}

We note that $\mathsf{SOCPA}$ and $\mathsf{S34A}$ require preprocessing to improve variable bounds on {the} $c$ and $s$ {variables as developed in Section \ref{sec:bounding}.} 
This process is parallelized but still it constitutes a sizable portion in the computational cost of the method. 
Constraint generation procedures are also parallelized, since each separation problem is defined for a different cycle in the cycle basis. We use a Gaussian elimination based approach to construct a cycle basis proposed in \cite{kocuk2014switch}.  We repeat the constraint generation algorithm for five iterations or terminate when there are no cuts to be added.

In the following, we compare the above four methods with the SDP relaxation based approaches in Section \ref{sec:comp2SDP} and with a recent quadratic convex relaxation approach in Section  \ref{sec:comp2QC}. We also show that the proposed methods are not dominated by nor dominate the SDP relaxations in Section \ref{sec:incomp2SDP}. Finally, in Section \ref{sec:comprob} we demonstrate the robustness of the proposed methods by solving randomly perturbed instances from the {standard IEEE instances}.

\subsection{Comparison to SDP {Relaxation} Based Methods} \label{sec:comp2SDP}
It is well known in the power systems literature that SDP relaxations have small duality gaps for the standard IEEE instances. However, the computational burden of SDP relaxations is typically very high. 
Chordal extensions and matrix completion type methods are used to significantly accelerate the solution time of the SDP relaxations. A publicly available implementation is called OPF Solver \citep{OPFSolver}. This package exploits the sparsity of underlying network to solve large-scale SDPs more efficiently as discussed in \cite{madani2014, madani2015}. In this section, we compare the accuracy and performance of the four proposed SOCP relaxation based methods to the SDP relaxation implemented in OPF Solver.

\subsubsection{Lower Bound and Computation Time Comparison}


We first compare the computation time and the lower bounds obtained by the SDP relaxation with those obtained by the four types of SOCP relaxations. Table \ref{all socp socpa} shows the results. Here,
``ratio" is defined as the lower bound of an SOCP relaxation divided by that of the SDP relaxation. We can see that the arctangent envelopes in $\mathsf{SOCPA}$ give non-trivial strengthening to the classic SOCP relaxation $\mathsf{SOCP}$. On the other hand, having the arctangent envelopes, the effect of the valid inequalities due to the McCormick relaxation of the cycle constraints is small. The SDP separation approach, $\mathsf{SSDP}$, is the most successful among the
four methods, which achieves the same lower bound as the SDP relaxation in nine instances and provides very tight bounds for the others ($99.96\%$ on average). In terms of computational time, $\mathsf{SOCPA}$, $\mathsf{S34A}$, and $\mathsf{SSDP}$ are roughly one order-of-magnitude faster than the SDP relaxation for large problems (2383-bus and above). We also note that OPF Solver does not support instances with reactive power cost functions, hence the case9Q and case30Q instances are solved using the standard rectangular SDP formulation. The largest instance case3375wp requires at least 3 hours to even construct the SDP model.

\begin{table}[h]\small
\caption{Comparison of lower bounds and computation time (NS: not supported, NA: not applicable).}\label{all socp socpa}
\begin{center}
\begin{tabular}{crrrrrrrrr}
\hline
           &     $\mathsf{SDP}$ & \multicolumn{ 2}{c}{$\mathsf{SOCP}$} & \multicolumn{ 2}{c}{$\mathsf{SOCPA}$} & \multicolumn{ 2}{c}{$\mathsf{S34A}$} & \multicolumn{ 2}{c}{$\mathsf{SSDP}$} \\
\cline{2-10}
  case &       time &      ratio &       time &      ratio &       time &      ratio &       time &     ratio &	time  \\
\hline
   6ww &       1.66 &     0.9937 &       0.02 &     0.9998 &       0.40 &     0.9999 &       0.43 &     1.0000 &       0.46 \\

     9 &       0.84 &     1.0000 &       0.02 &     1.0000 &       0.17 &     1.0000 &       0.18 &     1.0000 &       0.12 \\

    9Q &   NS &     1.0000 &       0.02 &     1.0000 &       0.18 &     1.0000 &       0.19 &     1.0000 &       0.12 \\

    14 &       1.07 &     0.9992 &       0.02 &     0.9992 &       0.41 &     0.9994 &       0.45 &     1.0000 &       0.64 \\

ieee30 &       1.84 &     0.9996 &       0.03 &     0.9996 &       0.78 &     0.9996 &       0.84 &     1.0000 &       1.15 \\

    30 &       2.19 &     0.9943 &       0.06 &     0.9963 &       0.95 &     0.9966 &       1.07 &     0.9993 &       1.22 \\

   30Q &   NS &     0.9753 &       0.07 &     0.9765 &       1.02 &     0.9769 &       1.11 &     1.0000 &       1.32 \\

    39 &       2.20 &     0.9998 &       0.04 &     0.9999 &       0.90 &     0.9999 &       0.99 &     1.0000 &       0.72 \\

    57 &       2.60 &     0.9994 &       0.04 &     0.9994 &       1.43 &     0.9994 &       1.47 &     1.0000 &       2.14 \\

   118 &       4.58 &     0.9976 &       0.11 &     0.9976 &       3.69 &     0.9984 &       4.83 &     0.9997 &       5.19 \\

   300 &       9.81 &     0.9985 &       0.21 &     0.9988 &       7.62 &     0.9989 &      10.40 &     1.0000 &       9.83 \\

2383wp &     682.86 &     0.9932 &       7.11 &     0.9949 &      92.83 &     0.9950 &     130.03 &     0.9984 &     101.31 \\

2736sp &     853.92 &     0.9970 &       5.14 &     0.9977 &      90.93 &     0.9976 &     163.80 &     0.9994 &      94.48 \\

2737sop &     792.25 &     0.9974 &       3.85 &     0.9979 &      95.28 &     0.9979 &     158.80 &     0.9997 &      78.70 \\

2746wop &    1138.06 &     0.9963 &       4.35 &     0.9971 &     102.37 &     0.9973 &     180.42 &     0.9995 &     109.65 \\

2746wp &     941.04 &     0.9967 &       5.79 &     0.9975 &     109.82 &     0.9975 &     186.31 &     0.9998 &     102.16 \\

3012wp &     746.08 &     0.9936 &       7.28 &     0.9946 &     143.10 &     0.9946 &     185.56 &     0.9974 &     109.19 \\

3120sp &     904.90 &     0.9955 &       7.33 &     0.9962 &     127.90 &     0.9965 &     196.05 &     0.9987 &     103.77 \\

3375wp &       $>$ 3hr &         NA &       8.25 &         NA &     149.03 &         NA &     422.35 &         NA &     133.62 \\

Average            & {\bf 380.37} & {\bf 0.9959} & {\bf 2.62} & {\bf 0.9968} & {\bf 48.88} & {\bf 0.9970} & {\bf 86.59} & {\bf 0.9996} & {\bf 45.04} \\
\hline
\end{tabular}  

\end{center}
\end{table}

\subsubsection{Upper Bound and Optimality Gap Comparison}
In this part, we compare the quality of the feasible solutions to the original OPF problem derived from relaxation solutions of our approaches to that of OPF Solver. Let us first describe our method of finding an OPF feasible solution. The procedure is simple: we use an optimal solution of one of the SOCP relaxations as a starting point to the nonlinear interior point solver IPOPT \citep{wachter}, which produces a locally optimal solution to the OPF problem. We {observed} empirically that independent of the relaxation we use, the method always converges to the same OPF solution.
We also note that this method gives the same OPF feasible solutions as MATPOWER 
{ 
and  ``flat start" to local solver,  that is, initializing IPOPT from $(c_{ij}, s_{ij})=(1,0)$ for all $(i,j)\in\mathcal{L}$ and $(c_{ii},\theta_{i}) = (1,0)$ for  all $i \in \mathcal{B}$ as proposed in \cite{Jabr08}. 
} 

\begin{table}[h]\small
\caption{Comparison of upper bounds and percentage optimality gap.}
\label{opt gap comp}
\begin{center}
\begin{tabular}{crrrrrrrrrrr}
\hline
           &             \multicolumn{ 3}{c}{$\mathsf{SDP}$} & \multicolumn{ 2}{c}{$\mathsf{SOCP}$} & \multicolumn{ 2}{c}{$\mathsf{SOCPA}$} & \multicolumn{ 2}{c}{$\mathsf{S34A}$} & \multicolumn{ 2}{c}{$\mathsf{SSDP}$} \\
\cline{2-12}
  case &     \%gap &       time &      ratio &     \%gap &       time &     \%gap &       time &     \%gap &       time &     \%gap &       time \\
\hline
   6ww &         NA &         NR &         NA &       0.63 &       0.13 &       0.02 &       0.48 &       0.01 &       0.45 &       0.00 &       0.53 \\

     9 &         NA &         NR &         NA &       0.00 &       0.04 &       0.00 &       0.19 &       0.00 &       0.20 &       0.00 &       0.17 \\

    9Q &         NA &         NR &         NA &       0.04 &       0.04 &       0.04 &       0.20 &       0.04 &       0.21 &       0.04 &       0.17 \\

    14 &       0.00 &       4.49 &     1.0000 &       0.08 &       0.05 &       0.08 &       0.44 &       0.06 &       0.48 &       0.00 &       0.68 \\

ieee30 &         NA &         NR &         NA &       0.04 &       0.07 &       0.04 &       0.83 &       0.04 &       0.88 &       0.00 &       1.20 \\

    30 &       0.00 &       6.54 &     1.0000 &       0.57 &       0.12 &       0.37 &       1.01 &       0.34 &       1.13 &       0.07 &       1.28 \\

   30Q &         NA &         NR &         NA &       2.48 &       0.11 &       2.35 &       1.07 &       2.32 &       1.16 &       0.00 &       1.36 \\

    39 &       0.01 &       5.09 &     1.0000 &       0.02 &       0.10 &       0.01 &       0.96 &       0.01 &       1.05 &       0.01 &       0.78 \\

    57 &       0.00 &       6.68 &     1.0000 &       0.06 &       0.11 &       0.06 &       1.50 &       0.06 &       1.55 &       0.00 &       2.22 \\

   118 &       0.00 &      11.16 &     1.0000 &       0.25 &       0.27 &       0.24 &       3.86 &       0.16 &       5.00 &       0.03 &       5.34 \\

   300 &       0.00 &      22.65 &     1.0000 &       0.15 &       0.62 &       0.12 &       8.04 &       0.11 &      10.83 &       0.00 &      10.33 \\

2383wp &       0.68 &     911.47 &     0.9969 &       1.05 &      21.39 &       0.89 &     104.71 &       0.88 &     145.29 &       0.54 &     124.34 \\

2736sp &       0.03 &    1181.09 &     0.9997 &       0.30 &      16.15 &       0.23 &      97.37 &       0.24 &     170.92 &       0.06 &     114.81 \\

2737sop &       0.00 &    1093.29 &     1.0000 &       0.26 &      12.05 &       0.21 &     102.27 &       0.21 &     167.59 &       0.03 &     103.81 \\

2746wop &       0.01 &    1470.10 &     0.9999 &       0.37 &       9.19 &       0.29 &     108.53 &       0.27 &     186.91 &       0.05 &     138.39 \\

2746wp &       0.04 &    1251.95 &     0.9996 &       0.33 &      14.08 &       0.25 &     116.18 &       0.25 &     193.91 &       0.02 &     124.07 \\

3012wp &       0.81 &    1314.16 &     0.9934 &       0.79 &      19.65 &       0.70 &     154.72 &       0.70 &     195.56 &       0.41 &     134.19 \\

3120sp &       0.93 &    1633.28 &     0.9916 &       0.54 &      16.14 &       0.47 &     137.70 &       0.44 &     206.20 &       0.22 &     121.77 \\

3375wp &         NA &       $>$3hr &         NA &       0.26 &      18.66 &       0.24 &     158.21 &       0.23 &     431.87 &       0.13 &     157.20 \\

 Average           & {\bf 0.19} & {\bf 685.53} & {\bf 0.9985} & {\bf 0.43} & {\bf 6.79} & {\bf 0.35} & {\bf 52.54} & {\bf 0.34} & {\bf 90.59} & {\bf 0.08} & {\bf 54.88} \\
\hline
\end{tabular}  
\end{center}
\end{table}

OPF Solver utilizes the SDP relaxation solution to obtain OPF feasible solutions. When the optimal matrix variable is rank optimal, e.g., rank one in the SDP relaxation in the real domain \eqref{eq:RSDPreal}, a vector of feasible voltages $e,f$ can be easily derived by 
computing the leading eigenvalue and the corresponding eigenvector of the SDP optimal matrix. 
However, when the rank is greater than one, it is difficult to put a physical meaning to the relaxation solution. OPF Solver uses a penalization approach to reduce the rank of the matrices in order to  obtain nearly feasible solutions to OPF. In particular, the reactive power dispatch and the total apparent power on some lines are penalized with certain penalty coefficients. Empirical results show that 
these coefficients are problem dependent and fine-tuning seems to be essential to obtain high quality feasible solutions. In our comparison, we use the suggested penalty parameters in \cite{madani2014} and exclude the computational burden of fine-tuning these parameters.

We compare the OPF feasible solutions found by our methods against the nearly feasible solutions obtained by OPF Solver. The results are shown in Table \ref{opt gap comp}.  Here,
``ratio" is calculated as the objective cost of an OPF feasible solution of our methods divided by that of OPF Solver. A ratio less than $1$ means our approach produces a better OPF feasible solution than OPF Solver.
The percentage optimality gap, ``\%gap'', is calculated as $\text{\%gap} = 100 \times \frac{z^{\text{UB}} - z^{\text{LB}}}{z^{\text{UB}}}$,
where $z^{\text{UB}}$ is the objective cost of an OPF feasible solution derived from a relaxation, and $z^{\text{LB}}$ is the optimal objective cost of this relaxation. The total computation time, reported as ``time'', includes the time solving the corresponding relaxation and deriving a feasible solution to OPF.
We observe that $\mathsf{SSDP}$ significantly closes the optimality gap to $0.08\%$ or $99.92\%$ to the global optimum on average, improving over the SDP relaxation's $0.19\%$. The ratio of upper bounds is less than $1$ for large systems, which implies 
the quality of the OPF feasible solutions obtained by the penalization method in OPF Solver are not as good as our approaches, even though best known penalty parameters are used. 
{
The reason is that the penalization method does not produce locally optimal solutions. This issue may perhaps be fixed by applying a local solver to improve the  solution obtained from penalization method at the cost of converting the optimal matrix variable to a vector of voltages and calling a local solver.
}
We also note that computing a feasible solution from the SDP relaxation is rather difficult, demonstrated by the large computational time of the $\mathsf{SDP}$ column, whereas $\mathsf{SOCP}$ is about two orders of magnitude faster than $\mathsf{SDP}$, and $\mathsf{SOCPA, S34A, SSDP}$ are roughly one order of magnitude faster.  

{
We also compare SDP bound with the feasible solution found by our SOCP based methods  and calculate the percentage optimality gap.
 Under an optimality threshold of 0.01\%, we observe that SDP is tight for 14 instances out of 19 (the gaps for cases 9Q, 2383wp, 3012wp, 3120sp and 3375wp are respectively 0.04\%, 0.37\%, 0.15\%, 0.09\%, NA).
We note that our SOCP based relaxations are not as successful according to this comparison. $\mathsf{SOCP}$, $\mathsf{SOCPA}$, $\mathsf{S34A}$ and   $\mathsf{SSDP}$ are tight for 1, 2, 3 and 8 instances, respectively. Nevertheless,  we remind the reader that SOCP based methods have small optimality gap (e.g. 0.08\% on the average for $\mathsf{SSDP}$) as can be seen from Table \ref{opt gap comp}.
}

\subsection{Comparison to Other SOCP Based Methods} \label{sec:comp2QC}

Now, we compare the strength of our SOCP based relaxations to other similar methods. A recent work utilizing SOCP relaxations is \cite{coffrin2015}, in which a Quadratic Convex (QC) relaxation of OPF is proposed. It is empirically observed that the phase angles of neighboring buses in a power network are usually close to each other in OPF problems and the QC relaxation is specialized to take advantage of this observation. However, very tight angle bounds are typically not available in practice and choosing very small angles may restrict the feasible region of the OPF problem. In this regards, we remind the reader that our proposed methods do not depend on the availability of such tight angle bounds and our methods use a preprocessing procedure to obtain bounds on the $c$ and $s$ variables. Explicit angle difference bounds can be incorporated into the SOCP relaxations by addition of the following constraints for $(i,j)\in\mathcal{L}$,
\begin{equation}\label{angle diff theta cs}
 -\tan (\overline \theta_{ij}) c_{ij}\le  s_{ij} \le \tan (\overline \theta_{ij}) c_{ij} 
\quad \text{and} \quad 
 -\overline \theta_{ij} \le \theta_i - \theta_j  \le \overline \theta_{ij},
\end{equation}
where $\overline \theta_{ij}$ is the maximum absolute difference between phase angles at buses $i$ and $j$. We also note that although the bounding techniques in Section \ref{sec:bounding} and the arctangent envelopes in $\mathsf{SOCPA, S34A}$ may be adapted to exploit the availability of such bounds, we choose not to do so in the experiments.

In Table \ref{opt gap comp nesta}, we compare the percentage optimality gaps of all the NESTA instances {obtained} by our methods and those achieved by the QC approach reported in \cite{coffrin2015}. The percentage optimality gap is defined the same as ``\%gap'' in the previous section. For the QC results, only instances with an optimality gap more than $1\%$ are reported in \cite{coffrin2015}. Those instances of optimal gaps less than $1\%$ are indicated by blanks in Table \ref{opt gap comp nesta}. The average optimality gap for the QC approach is taken over instances with reported values. The NESTA library has three types of instances, namely, the typical operating conditions, congested operating conditions, and small angle difference conditions (\cite{coffrin2015}). 

From Table \ref{opt gap comp nesta} we have the following observations.
	\begin{enumerate}
		\item For instances from Typical Operating Conditions, each of our three strong SOCP relaxations dominates QC for all instances, except for the 3-bus instance 3lmbd $\mathsf{SOCPA}$ has an optimality gap $1.25\%$ comparing to QC's $1.24\%$. For all the instances where QC achieves an optimality gap less than $1\%$, SOCP relaxations also achieve less than $1\%$ gaps, except for the 1460wp instance, for which a gap slightly higher than $1\%$ is obtained by the strong SOCP relaxations. The $\mathsf{SSDP}$ approach significantly outperforms QC in all instances and on average achieves $1.82\%$ gap versus QC's $5.17\%$. 
		\item A similar picture holds for the Congested Operating Conditions, where the three strong SOCP relaxations dominate QC for all instances, except for $\mathsf{SOCPA}$ on the 3-bus instance 3lmbd. All three strong SOCP relaxations achieve less than $1\%$ optimality gaps on all instances that QC achieves less than $1\%$ gaps. $\mathsf{SSDP}$ again has the best performance and significantly outperforms QC.
		\item For the instances from Small Angle Difference Conditions, {which is a condition that is most suitable for QC}, QC only dominates $\mathsf{SSDP}$ 4 times out of 19 instances reported in (\cite{coffrin2015}), and {QC is better than all of strong SOCP relaxation in 3 out of the 19 instances}. In terms of the average optimality gap, both SOCPA and S34A outperform QC.
	\end{enumerate} 


%
%
%
Computational costs of our methods for NESTA instances are provided in Appendix \ref{app:nesta time}.

\begin{landscape}
\begin{table}[]
\caption{Comparison of percentage optimality gap for NESTA instances.
}
\label{opt gap comp nesta}
\begin{center}

\begin{tabular}{c|rrrrr|rrrrr|rrrrr|}
\cline{2-16}
           &                                  \multicolumn{ 5}{|c}{Typical Operating Conditions} &                                \multicolumn{ 5}{|c}{Congested Operating Conditions} &                              \multicolumn{ 5}{|c|}{Small Angle Difference Conditions} \\
\cline{2-16}
  case         &    $\mathsf{SOCP}$  &      $\mathsf{SOCPA}$ &   $\mathsf{S34A}$ &   $\mathsf{SSDP}$ & QC  &      $\mathsf{SOCP}$  &      $\mathsf{SOCPA}$ &   $\mathsf{S34A}$ &   $\mathsf{SSDP}$ & QC &       $\mathsf{SOCP}$  &      $\mathsf{SOCPA}$ &   $\mathsf{S34A}$ &   $\mathsf{SSDP}$ & QC  \\
\hline

     3lmbd &       1.32 &       1.25 &       0.97 &       0.43 &       1.24 &       3.30 &       1.97 &       1.20 &       1.31 &       1.83 &       4.28 &       2.33 &       1.51 &       2.13 & { 1.24} \\

       4gs &       0.00 &       0.00 &       0.00 &       0.01 &            &       0.65 &       0.16 &       0.12 &       0.00 &            &       4.90 &       0.42 &       0.02 &       0.14 &       0.81 \\

      5pjm &      14.54 &      14.47 &      14.26 &       6.22 &      14.54 &       0.45 &       0.11 &       0.06 &       0.00 &            &       3.61 &       0.45 &       0.34 &       0.01 &       1.10 \\

       6ww &       0.63 &       0.02 &       0.01 &       0.00 &            &      13.33 &       0.35 &       0.14 &       0.00 &      13.14 &       0.80 &       0.02 &       0.01 &       0.00 &            \\

     9wscc &       0.00 &       0.00 &       0.00 &       0.00 &            &       0.00 &       0.00 &       0.00 &       0.00 &            &       1.50 &       0.43 &       0.37 &       0.01 &       0.41 \\

    14ieee &       0.11 &       0.11 &       0.07 &       0.00 &            &       1.35 &       1.32 &       1.32 &       0.00 &       1.35 &       0.07 &       0.06 &       0.06 &       0.00 &            \\

    29edin &       0.14 &       0.08 &       0.05 &       0.00 &            &       0.44 &       0.40 &       0.36 &       0.03 &            &      34.47 &      25.94 &      21.06 &      31.33 & { 20.57} \\

      30as &       0.06 &       0.05 &       0.05 &       0.00 &            &       4.76 &       2.02 &       1.89 &       1.72 &       4.76 &       9.16 &       2.43 &       2.36 &       0.95 &       3.07 \\

     30fsr &       0.39 &       0.23 &       0.23 &       0.03 &            &      45.97 &      42.22 &      41.85 &      40.28 &      45.97 &       0.62 &       0.33 &       0.27 &       0.12 &      \\

    30ieee &      15.65 &       5.24 &       4.79 &       0.00 &      15.44 &       0.99 &       0.86 &       0.85 &       0.08 &            &       5.87 &       2.07 &       1.98 &       0.00 &        3.95      \\

    39epri &       0.05 &       0.02 &       0.02 &       0.01 &            &       2.99 &       0.77 &       0.77 &       0.00 &       2.97 &       0.11 &       0.09 &       0.09 &       0.09 &            \\

    57ieee &       0.06 &       0.06 &       0.06 &       0.00 &            &       0.21 &       0.21 &       0.20 &       0.13 &            &       0.11 &       0.09 &       0.09 &       0.05 &            \\

   118ieee &       2.10 &       1.12 &       0.94 &       0.25 &       1.75 &      44.19 &      40.18 &      38.22 &      39.09 &      44.03 &      12.88 &       7.77 &       7.32 &       9.50 &       8.30 \\

   162ieee &       4.19 &       3.99 &       3.95 &       3.50 &       4.17 &       1.52 &       1.44 &       1.43 &       1.20 &       1.51 &       7.06 &       5.94 &       5.81 &       6.36 &       6.88 \\

   189edin &       0.22 &       0.22 &       0.22 &       0.07 &            &       5.59 &       3.34 &       3.33 &       0.22 &       5.56 &       2.27 &       2.21 &       2.25 &       1.23 &       2.24 \\

   300ieee &       1.19 &       0.78 &       0.71 &       0.30 &       1.18 &       0.85 &       0.51 &       0.47 &       0.15 &            &       1.27 &       0.77 &       0.70 &       0.33 &       1.16 \\

    1460wp &       1.22 &       1.18 &       1.18 &       1.04 &            &       1.10 &       0.98 &       0.84 &       0.68 &            &       1.37 &       1.33 &       1.32 &       1.22 &            \\

  2224edin &       6.22 &       4.30 &       4.25 &       4.60 &       6.16 &       3.16 &       2.51 &       2.43 &       2.58 &       3.15 &       6.43 &       3.91 &       3.87 &       4.80 &       5.79 \\

    2383wp &       1.06 &       0.87 &       0.87 &       0.54 &       1.04 &       1.12 &       0.91 &       0.87 &       0.52 &       1.12 &       4.01 &       2.92 &       2.80 &       2.82 &       2.97 \\

    2736sp &       0.30 &       0.21 &       0.20 &       0.08 &            &       1.33 &       1.14 &       1.12 &       0.91 &       1.32 &       2.34 &       1.86 &       1.86 &       1.92 &       2.01 \\

   2737sop &       0.26 &       0.20 &       0.20 &       0.03 &            &       1.06 &       0.86 &       0.86 &       0.54 &       1.05 &       2.43 &       2.23 &       2.23 &       1.97 & { 2.21} \\

   2746wop &       0.37 &       0.28 &       0.27 &       0.06 &            &       0.49 &       0.35 &       0.34 &       0.17 &            &       2.94 &       2.30 &       2.31 &       2.60 & { 1.83} \\

    2746wp &       0.32 &       0.22 &       0.22 &       0.03 &            &       0.58 &       0.34 &       0.34 &       0.07 &            &       2.44 &       1.68 &       1.67 &       1.83 &       2.48 \\

    3012wp &       1.04 &       0.90 &       0.89 &       0.50 &       1.01 &       1.25 &       0.90 &       0.89 &       0.58 &       1.24 &       2.14 &       2.00 &       1.96 &       1.54 & { 1.92} \\

    3120sp &       0.56 &       0.45 &       0.44 &       0.23 &            &       3.03 &       2.78 &       2.78 &       2.34 &       3.02 &       2.79 &       2.60 &       2.57 &       2.19 &       2.56 \\

    3375wp &       0.53 &       0.47 &       0.46 &       0.29 &            &       0.82 &       0.64 &       0.64 &       0.39 &            &       0.53 &       0.45 &       0.45 &       0.28 &            \\
Average &    {\bf 5.26} &       {\bf 3.66} &       {\bf 3.51} &       {\bf 1.82} &       {\bf 5.17} &      {\bf 8.93} &       {\bf 6.85} &       {\bf 6.62} &       {\bf 6.14} &       {\bf 8.80} &    {\bf 5.94} &       {\bf 3.70} &       {\bf 3.36} &       {\bf 4.38} &       {\bf 3.76} \\
\hline
\end{tabular}  
\end{center}
\end{table}
\end{landscape}

\subsection{Incomparability of the Proposed Methods with SDP Relaxation} \label{sec:incomp2SDP}
\label{incomparable}

We now prove the incomparability of our proposed methods with the SDP relaxation as stated in Propositions \ref{cycle incomp} and \ref{arctan incomp} using three specific instances from the NESTA archive. To start with, let $\mathsf{S34}$ denote a variant of $\mathsf{S34A}$ without the arctangent envelopes. 
The percentage optimality gaps of these instances are presented in Table \ref{incomp example}. First of all, note that the SDP relaxation is not dominated by $\mathsf{S34}$ or $\mathsf{SOCPA}$ as the 5-bus case 5pjm demonstrates. And SDP does not dominate $\mathsf{S34}$ due to case 3lmbd, which proves Proposition \ref{cycle incomp}; also SDP does not dominate $\mathsf{SOCPA}$ due to case 29edin, which proves Proposition \ref{arctan incomp}. These instances also show that $\mathsf{S34}$ and $\mathsf{SOCPA}$ are incomparable.

\begin{table}[htbp]
\captionsetup{width=.88\textwidth}\caption{Percentage optimality gap of three instances from NESTA with small angle difference conditions. 
}
\label{incomp example}
\begin{center}
\begin{tabular}{crrr}
\hline
 case     &       $\mathsf{S34}$ &     $\mathsf{SOCPA}$  &   $\mathsf{SDP}$ \\
\hline
   5pjm	&    0.40 &      0.45 &      0.00  \\

  3lmbd&     1.53 &       2.33 &     2.06\\

    29edin &     31.39 &     25.94&      28.44 \\
\hline
\end{tabular}  
\end{center}
\end{table}

\subsection{Robustness of the Proposed Methods} \label{sec:comprob}
We test the robustness of our methods by solving perturbed instances to the standard IEEE benchmarks. In particular, load values are randomly perturbed 5\% to obtain 10 similar and realistic instances. Results in Table \ref{perturb} show that our methods consistently provide provably good solutions and tight relaxations for OPF problem.

\begin{table}[H]
\caption{Average percentage optimality gaps of perturbed IEEE standard benchmarks.} 
\label{perturb}
\begin{center}
\begin{tabular}{crrrrrrrr}
\hline
           & \multicolumn{ 2}{c}{$\mathsf{SOCP}$} & \multicolumn{ 2}{c}{$\mathsf{SOCPA}$} & \multicolumn{ 2}{c}{$\mathsf{S34A}$} & \multicolumn{ 2}{c}{$\mathsf{SSDP}$} \\
\cline{2-9}
      case &      \%gap &       time &       \%gap &       time &       \%gap &       time &       \%gap &       time \\
\hline
       6ww &       0.62 &       0.06 &       0.02 &       0.26 &       0.01 &       0.32 &       0.00 &       0.46 \\

         9 &       0.00 &       0.04 &       0.00 &       0.19 &       0.00 &       0.20 &       0.00 &       0.11 \\

        9Q &       0.09 &       0.04 &       0.09 &       0.19 &       0.09 &       0.20 &       0.09 &       0.12 \\

        14 &       0.08 &       0.05 &       0.08 &       0.38 &       0.06 &       0.41 &       0.00 &       0.61 \\

    ieee30 &       0.04 &       0.06 &       0.04 &       0.78 &       0.04 &       0.81 &       0.00 &       1.03 \\

        39 &       0.03 &       0.09 &       0.01 &       0.91 &       0.01 &       0.99 &       0.00 &       0.82 \\

        57 &       0.07 &       0.11 &       0.07 &       1.45 &       0.07 &       1.51 &       0.00 &       1.93 \\

       118 &       0.25 &       0.30 &       0.25 &       3.64 &       0.17 &       5.12 &       0.04 &       5.04 \\

       300 &       0.63 &       0.66 &       0.60 &       7.90 &       0.58 &      13.71 &       0.33 &      10.05 \\

    2736sp &       0.30 &      12.67 &       0.23 &     110.42 &       0.23 &     201.81 &       0.05 &     120.92 \\

   2737sop &       0.26 &      11.98 &       0.22 &     108.89 &       0.22 &     188.52 &       0.03 &      92.18 \\

   2746wop &       0.38 &       9.46 &       0.30 &     114.89 &       0.28 &     215.16 &       0.06 &     115.31 \\

    2746wp &       0.32 &      12.43 &       0.25 &     125.95 &       0.25 &     217.88 &       0.05 &     117.24 \\

    3012wp &       0.81 &      16.77 &       0.71 &     125.77 &       0.71 &     167.80 &       0.43 &     116.88 \\

    3120sp &       0.53 &      16.19 &       0.44 &     134.98 &       0.44 &     180.07 &       0.25 &     115.47 \\

    3375wp &       0.26 &      19.30 &       0.24 &     179.34 &       0.23 &     481.92 &       0.19 &     161.67 \\

   Average    & {\bf 0.29} & {\bf 6.26} & {\bf 0.22} & {\bf 57.25} & {\bf 0.21} & {\bf 104.78} & {\bf 0.10} & {\bf 53.74} \\

\hline
\end{tabular}  
\end{center}
\end{table}

We note that we have not reported 3 instances, namely cases 30, 30Q and 2383wp, in Table  \ref{perturb}. In fact,
for case2383wp, all the instances are proven to be infeasible  by SOCP. On the other hand, for cases 30 and 30Q, we are able to find feasible solutions for only one instance. For the remaining nine instances,  six of them are proven to be infeasible by $\mathsf{SOCPA}$, $\mathsf{S34A}$, and $\mathsf{SSDP}$ but not by $\mathsf{SOCP}$. We also note that eight instances are proven to be infeasible by SDP.

%
%
%
%
%

\section{Conclusions and Future Work}
\label{sec:conc}
In this paper, we study the OPF problem, a fundamental optimization problem in electric power system analysis. 
We have the following main conclusions:
\begin{enumerate}
\item The proposed strong SOCP relaxations offer a computationally attractive alternative to the SDP relaxations for practically solving large-scale OPF problems. The lower bounds obtained by the strong SOCP relaxations are extremely close to those of the SDP relaxations, and are not always dominated by the SDP relaxations, but within a computation time that is an order of magnitude faster than the latter. 
In case tight bounds on phase angles are known, the strong SOCP relaxation involving arctan linearization together with McCormick constraints from the cycle decompositions is recommended. In case explicit bounds on phase angles are not known, the third approach of SOCP with SDP separation is recommended.


\item The proposed SOCP relaxation produces a solution that can be conveniently used as a good warm start for a local solver, such as IPOPT. In comparison, recovering a feasible solution from the SDP relaxations is a computationally challenging task, even when the SDP relaxation is tight. 

\item The proposed SOCP relaxations provide stronger bounds than existing quadratic relaxation approaches \cite{coffrin2014, coffrin2015} on most instances. 

\end{enumerate}

There are three future research directions we would like to pursue. Firstly, there is a need to  implement a spatial branch-and-bound algorithm to obtain globally optimal solutions to OPF since there are still some instances that  are not  solved to global optimality although our SOCP based methods close significant gap.
Secondly,  some of the ideas from this paper can be applied to solve AC Optimal Transmission Switching  or multiperiod OPF Problems. 
Finally, on the theory side, we would like to investigate the existence of verifiable sufficient conditions for the exactness of SOCP based approaches introduced in this paper.

\appendix
\section{Times for NESTA Instances}
\label{app:nesta time}

Computational costs of different relaxation methods for NESTA instances are provided in Table~\ref{nesta time}.

\begin{landscape}
\begin{table}[htbp]
\caption{Computational costs of different methods for NESTA instances.
}
\label{nesta time}
\begin{center}
\begin{tabular}{r|rrrr|rrrr|rrrr|}
\cline{2-13}
           &                     \multicolumn{ 4}{|c}{Typical Operating Conditions} &                   \multicolumn{ 4}{|c}{Congested Operating Conditions} &                 \multicolumn{ 4}{|c|}{Small Angle  Conditions} \\
\cline{2-13}
           &      $\mathsf{SOCP}$  &      $\mathsf{SOCPA}$ &   $\mathsf{S34A}$ &   $\mathsf{SSDP}$  &      $\mathsf{SOCP}$  &      $\mathsf{SOCPA}$ &   $\mathsf{S34A}$ &   $\mathsf{SSDP}$  &      $\mathsf{SOCP}$  &      $\mathsf{SOCPA}$ &   $\mathsf{S34A}$ &   $\mathsf{SSDP}$  \\
\hline
     3lmbd &       0.02 &       0.13 &       0.31 &       0.21 &       0.03 &       0.09 &       0.14 &       0.11 &       0.02 &       0.14 &       0.32 &       0.13 \\

       4gs &       0.04 &       0.11 &       0.10 &       0.08 &       0.02 &       0.09 &       0.11 &       0.10 &       0.02 &       0.09 &       0.14 &       0.14 \\

      5pjm &       0.02 &       0.13 &       0.15 &       0.21 &       0.02 &       0.14 &       0.16 &       0.23 &       0.02 &       0.12 &       0.13 &       0.21 \\

       6ww &       0.02 &       0.22 &       0.24 &       0.39 &       0.02 &       0.22 &       0.24 &       0.51 &       0.02 &       0.21 &       0.25 &       0.39 \\

     9wscc &       0.02 &       0.17 &       0.18 &       0.11 &       0.02 &       0.18 &       0.19 &       0.13 &       0.02 &       0.17 &       0.19 &       0.17 \\

    14ieee &       0.02 &       0.45 &       0.48 &       0.67 &       0.02 &       0.42 &       0.46 &       0.66 &       0.02 &       0.40 &       0.44 &       0.65 \\

    29edin &       0.20 &       3.89 &       4.86 &       2.44 &       0.14 &       3.61 &       4.53 &       2.50 &       0.14 &       3.42 &       5.01 &       2.56 \\

      30as &       0.04 &       1.00 &       1.08 &       1.09 &       0.04 &       0.93 &       1.16 &       1.14 &       0.04 &       0.99 &       1.14 &       1.15 \\

     30fsr &       0.03 &       0.99 &       1.06 &       1.11 &       0.04 &       0.91 &       1.05 &       1.15 &       0.04 &       0.88 &       1.16 &       1.14 \\

    30ieee &       0.04 &       0.88 &       1.03 &       1.08 &       0.04 &       0.91 &       1.00 &       1.14 &       0.04 &       0.89 &       0.97 &       1.14 \\

    39epri &       0.04 &       0.89 &       0.97 &       0.64 &       0.04 &       0.95 &       0.96 &       0.86 &       0.04 &       0.90 &       0.98 &       0.82 \\

    57ieee &       0.08 &       2.04 &       2.09 &       2.17 &       0.07 &       1.71 &       1.94 &       2.33 &       0.07 &       1.84 &       2.06 &       2.10 \\

   118ieee &       0.25 &       4.98 &       5.57 &       5.97 &       0.19 &       4.75 &       5.84 &       5.79 &       0.33 &       4.86 &       5.92 &       6.34 \\

   162ieee &       0.25 &       8.97 &      12.14 &      11.37 &       0.23 &       9.30 &      13.50 &      11.93 &       0.34 &       9.04 &      13.65 &      11.53 \\

   189edin &       0.40 &       4.84 &       5.79 &       2.89 &       0.35 &       4.89 &       5.38 &       3.21 &       0.43 &       5.35 &       6.34 &       3.31 \\

   300ieee &       0.47 &       9.93 &      14.96 &      11.58 &       0.44 &      10.54 &      14.26 &      11.11 &       0.55 &       9.74 &      13.11 &      11.76 \\

    1460wp &       2.96 &      63.26 &      83.04 &      45.46 &       3.88 &      74.73 &     117.83 &      52.09 &       3.64 &      78.12 &     114.10 &      42.69 \\

  2224edin &       5.99 &     118.38 &     167.43 &      95.95 &       8.72 &     130.32 &     234.88 &     112.95 &       6.65 &     110.46 &     159.24 &      96.03 \\

    2383wp &       7.59 &     120.68 &     165.48 &     117.79 &       8.39 &     132.29 &     296.21 &     127.18 &       7.08 &     117.14 &     205.04 &     108.82 \\

    2736sp &       5.67 &     119.35 &     221.79 &     130.44 &      10.28 &     151.94 &     294.40 &     191.63 &       5.11 &     112.51 &     150.25 &     121.21 \\

   2737sop &       4.61 &     114.62 &     149.41 &      95.37 &       9.59 &     148.50 &     218.83 &     142.93 &       3.90 &     108.21 &     143.75 &      85.97 \\

   2746wop &       5.22 &     118.63 &     215.98 &     116.85 &       6.48 &     130.92 &     260.75 &     142.77 &       4.49 &     113.43 &     221.35 &     109.04 \\

    2746wp &       6.84 &     123.01 &     203.85 &     117.24 &       7.42 &     136.72 &     229.48 &     137.48 &       5.65 &     116.49 &     209.42 &     127.23 \\

    3012wp &       7.12 &     142.06 &     242.37 &     123.21 &      13.30 &     184.93 &     373.88 &     179.50 &       7.97 &     162.75 &     370.68 &     132.40 \\

    3120sp &       7.72 &     147.32 &     255.43 &     125.04 &      11.34 &     164.37 &     310.54 &     150.02 &       9.38 &     177.76 &     337.48 &     137.24 \\

    3375wp &       9.15 &     184.14 &     306.22 &     169.13 &      18.31 &     239.98 &     600.27 &     247.87 &      10.66 &     208.05 &     436.30 &     194.67 \\
\hline
\end{tabular}  

\end{center}
\end{table}
\end{landscape}

\section*{Acknowledgments}
We wish to thank Dr. Carleton Coffrin of NICTA for providing us NESTA 0.3.0 archive. {We really appreciate his prompt help in this matter.}

\bibliographystyle{ormsv080}
\bibliography{references}


\end{document}